\documentclass[leqno,12pt]{article}

\usepackage{graphicx}
\usepackage{fancyhdr}
\usepackage{amsmath}
\usepackage{amsthm}
\usepackage{amsfonts}
\newtheorem{theorem}{Theorem}[section]
\newtheorem{lemma}[theorem]{Lemma}
\newtheorem{corollary}[theorem]{Corollary}
\theoremstyle{definition}
\newtheorem{definition}[theorem]{Definition}
\newtheorem{remark}[theorem]{Remark}
\newtheorem{notation}[theorem]{Notation}

\def\address#1{{\center{#1}}}
\date{}
\def\m@th{\mathsurround=0pt}
\def\eqal#1{\null\,\vcenter{\openup\jot\m@th
\ialign{\strut\hfil$\displaystyle{##}$&&$\displaystyle{{}##}$\hfil
 \crcr#1\crcr}}\,}
\def\matrix#1{\null\,\vcenter{\normalbaselines\m@th
 \ialign{\hfil$##$\hfil&&\quad\hfil$##$\hfil\crcr
 \mathstrut\crcr\noalign{\kern-\baselineskip}
 #1\crcr\mathstrut\crcr\noalign{\kern-\baselineskip}}}\,}
\def\N{{\Bbb N}}
\def\R{{\Bbb R}}
\def\divv{{\rm div}\,}
\def\curl{{\rm curl}\,}

\numberwithin{equation}{section}

\title{On regularity estimates for axially symmetric Navier--Stokes equations in a cylinder and \\the critical-wedge occurrence problem}
\author{Wies\l aw J. Grygierzec$^{(1)*}$, Wojciech M. Zaj\c{a}czkowski$^{(2,3)}$}

\newcommand\blfootnote[1]{%
  \begingroup
  \renewcommand\thefootnote{}\footnote{#1}%
  \addtocounter{footnote}{-1}%
  \endgroup
}

\begin{document}
\input amssym.def
\input amssym.tex
\maketitle
\blfootnote{*Corresponding author.
Email: \texttt{wieslaw.grygierzec@urk.edu.pl}}

\address{
  \begin{tabular}{c}
    $^1$Department of Statistics and Social Policy,\\
    University of Agriculture in Krak\'ow, Al. Mickiewicza 21,\\
    31-120 Krak\'ow, Poland.\\
    e-mail: \texttt{wieslaw.grygierzec@urk.edu.pl}
  \end{tabular}\\[2.5ex]
  \begin{tabular}{c}
    $^2$Institute of Mathematics, Polish Academy of Sciences (emeritus professor),\\
    \'Sniadeckich 8, 00-656 Warsaw, Poland\\
    e-mail: \texttt{wz@impan.pl}
  \end{tabular}\\[2.5ex]
  \begin{tabular}{c}
    $^3$Institute of Mathematics and Cryptology, Cybernetics Faculty,\\
    Military University of Technology,\\
    S. Kaliskiego 2, 00-908 Warsaw, Poland.
  \end{tabular}
}

\begin{abstract}
  We consider the axisymmetric Navier-Stokes equations in a finite cylinder $\Omega\subset\R^3$. We assume that $v_r$, $v_\varphi$, $\omega_\varphi$ vanish on the lateral part of boundary $\partial\Omega$ of the cylinder, and that $v_z$, $\omega_\varphi$, $\partial_zv_\varphi$ vanish on the top and bottom parts of the boundary $\partial\Omega$, where we used standard cylindrical coordinates, and we denoted by $\omega=\curl v$ the vorticity field. Our aim is to derive estimates for
  $$
    X_s(t) := \bigg\|{\omega_r\over r}\bigg\|_{V(\Omega_t)}+\bigg\|{\omega_\varphi\over r}\bigg\|_{V(\Omega_t)}.
  $$
  The original closure mechanism depends on the relation between the $L^s$ norm and the $L^\infty$ norm of the angular component $v_\varphi$. We identify a critical wedge in the corresponding phase geometry, namely the regime
  $$
  W_{A,c_0} := \left\{
  t\in(0,T):
  \|v_\varphi(t)\|_{L^s(\Omega)}>A,\quad
  \frac{\|v_\varphi(t)\|_{L^s(\Omega)}}{\|v_\varphi(t)\|_{L^\infty(\Omega)}}<c_0
  \right\},
  $$
  with the convention that the ratio is $+\infty$ when $\|v_\varphi(t)\|_{L^\infty(\Omega)}=0$.

  The main result is a conditional a priori estimate in which the possible loss of control is measured by a critical-wedge residual. More precisely, if $E_{W,s}$ denotes the positive part of the non-closable contribution of the nonlinear interaction
  $$
  \int_{\Omega_t}\frac{v_\varphi}{r}\Phi\Gamma\,dx\,dt',
  $$
  restricted to $W_{A,c_0}$, then
  $$
  X_s(t)
  \leq
  \Psi_{s,A,c_0}\left(
  \mathrm{data},
  \int_0^t E_{W,s}(\tau)\,d\tau
  \right),
  \qquad 0<t<T,
  $$
  for an increasing positive function $\Psi_{s,A,c_0}$. In particular, if the critical-wedge residual vanishes, for instance when the trajectory does not enter $W_{A,c_0}$, the original type of a priori estimate depending only on the data is recovered. Under additional regularity assumptions on the force and the initial velocity, the corresponding higher Sobolev estimate for $v$ and $\nabla p$ follows with the same conditional dependence.
\end{abstract}


\section{Introduction}\label{s1}

We are concerned with the 3D incompressible Navier-Stokes equations,
\begin{equation}\eqal{
    &\partial_tv-\nu\Delta v+v\cdot\nabla v+\nabla p=f,\cr
    &\divv v=0\quad {\rm in}\ \ \Omega^T,\cr}
  \label{1.1}
\end{equation}
under the axisymmetry constraint, where $\Omega^T\colon=\Omega\times(0,T)$, $T>0$, $v=v(x,t)\in\R^3$ denotes the velocity field, $p=p(x,t)\in\R$ denotes the pressure function, $f=f(x,t)\in\R^3$ denotes the external force field, $\nu>0$ denotes the viscosity, and $x=(x_1,x_2,x_3)$ denotes the Cartesian coordinates. As for $\Omega$ we focus on the case of a finite cylinder,
$$
  \Omega=\{x\in\R^3\colon x_1^2+x_2^2<R^2,|x_3|<a\},
$$
where $a,R>0$ are constants. We note that
$$
  S\colon=\partial\Omega=S_1\cup S_2,
$$
where
$$\eqal{
    &S_1=\{x\in\R^3\colon\sqrt{x_1^2+x_2^2}=R,x_3\in[-a,a]\},\cr
    &S_2=\{x\in\R^3\colon\sqrt{x_1^2+x_2^2}<R,x_3\in\{-a,a\}\}\cr}
$$
denote the lateral boundary and the top and bottom parts of the boundary respectively.

In order to state the boundary conditions stating our main result we use the cylindrical coordinates $r$, $\varphi$, $z$ defined by
$$
  x_1=r\cos\varphi,\quad x_2=r\sin\varphi,\quad x_3=z,
$$
and we will use standard cylindrical unit vectors, so that, for example
$$
  v=v_r\bar e_r+v_\varphi\bar e_\varphi+v_z\bar e_z.
$$
We will denote partial derivatives by using the subscript comma notation, e.g.
$$
  v_{r,z}\colon=\partial_zv_r.
$$
We assume the boundary conditions
\begin{equation}\eqal{
    &v_r=v_\varphi=\omega_\varphi=0\quad &{\rm on}\ \ S_1^T=S_1\times(0,T),\cr
    &v_z=\omega_\varphi=v_{\varphi,z}=0\quad &{\rm on}\ \ S_2^T=S_2\times(0,T),\cr}
  \label{1.2}
\end{equation}
where $\omega\colon=\curl v$ denotes the vorticity vector and we assume the initial condition
\begin{equation}
  v|_{t=0}=v_0,
  \label{1.3}
\end{equation}
where $v_0$ is a given divergence-free vector field satisfying the same boundary conditions.

We note that such boundary conditions have first appeared in the work of Ladyzhenskaya \cite{L}. In a sense, the boundary conditions (\ref{1.2}) are natural, since, when considering the vorticity-stream function formulation we need $\omega_\varphi|_S$. This together with the no-penetration condition naturally lead to~(\ref{1.2}).

The aim of this paper is to analyze a priori estimates for axially symmetric solutions to problem (\ref{1.1})--(\ref{1.3}), with particular emphasis on the quantities
\[
  \Phi=\frac{\omega_r}{r},\qquad \Gamma=\frac{\omega_\varphi}{r}.
\]
The main difficulty is related to the nonlinear interaction
\[
  \int_{\Omega_t}\frac{v_\varphi}{r}\Phi\Gamma\,dx\,dt',
\]
whose control depends on the distribution of the angular velocity component $v_\varphi$.

We collect all parameters of this paper in Notation \ref{n1.1}. They depend on initial data and forcing. The main results necessary for the proof of Theorem \ref{t1.2} are presented in Section \ref{s3}. To derive them we need $H^2-H^3$-elliptic estimates for the modified stream function (stream function divided by $r$) proved in Section \ref{s4}, energy estimates for the gradient of swirl are found in Section \ref{s5} and the crucial order reduction inequality is proved in Section \ref{s6}. Theorem \ref{t1.3} is proved in Section \ref{s7}. To prove it we need solvability of the Stokes system in Sobolev spaces with the mixed norm and theorems on traces for the Besov spaces with the mixed norm.

We will denote the swirl by
\begin{equation}
  u\colon=rv_\varphi.
  \label{1.4}
\end{equation}
Note that
\begin{equation}\eqal{
  &\omega_r=-v_{\varphi,z}=-{1\over r}u_{,z},\cr
  &\omega_\varphi=v_{r,z}-v_{z,r},\cr
  &\omega_z={1\over r}(rv_\varphi)_{,r}=v_{\varphi,r}+{v_\varphi\over r}={1\over r}u_{,r},\cr}
  \label{1.5}
\end{equation}
so that the boundary conditions (\ref{1.2}) imply in particular that
\begin{equation}\eqal{
    &\omega_r=v_{z,r}=u=0,\ \ \omega_z=v_{\varphi,r}\quad &{\rm on}\ \ S_1^T,\cr
    &\omega_r=v_{r,z}=\omega_{z,z}=v_{\varphi,z}=u_{,z}=0\quad &{\rm on}\ \ S_2^T.\cr}
  \label{1.6}
\end{equation}
The Navier-Stokes equations (\ref{1.1}) in cylindrical coordinates become
\begin{equation}\eqal{
  &v_{r,t}+v\cdot\nabla v_r-{v_\varphi^2\over r}-\nu\Delta v_r+\nu{v_r\over r^2}=-p_{,r}+f_r,\cr
  &v_{\varphi,t}+v\cdot\nabla v_\varphi+{v_r\over r}v_\varphi-\nu\Delta v_\varphi+\nu{v_\varphi\over r^2}=f_\varphi,\cr
  &v_{z,t}+v\cdot\nabla v_z-\nu\Delta v_z=-p_{,z}+f_z,\cr
  &(rv_r)_{,r}+(rv_z)_{,z}=0,\cr}
  \label{1.7}
\end{equation}
where
$$
  v\cdot\nabla=(v_r\bar e_r+v_z\bar e_z)\cdot\nabla=v_r\partial_r+v_z\partial_z,\quad
  \Delta u={1\over r}(ru_{,r})_{,r}+u_{,zz}.
$$
On the other hand, the vorticity formulation becomes
\begin{equation}\eqal{
  &\omega_{r,t}+v\cdot\nabla\omega_r-\nu\Delta\omega_r+\nu{\omega_r\over r^2}=\omega_rv_{r,r}+\omega_zv_{r,z}+F_r,\cr
  &\omega_{\varphi,t}+v\cdot\nabla\omega_\varphi-{v_r\over r}\omega_\varphi-\nu\Delta\omega_\varphi+\nu{\omega_\varphi\over r^2}={2\over r}v_\varphi v_{\varphi,z}+F_\varphi,\cr
  &\omega_{z,t}+v\cdot\nabla\omega_z-\nu\Delta\omega_z=\omega_rv_{z,r}+\omega_zv_{z,z}+F_z,\cr}
  \label{1.8}
\end{equation}
where $F\colon=\curl f$ and the swirl is a solution to the problem
\begin{equation}\eqal{
  &u_{,t}+v\cdot\nabla u-\nu\Delta u+{2\nu\over r}u_{,r}=rf_\varphi\colon=f_0,\cr
  &u=0\quad &{\rm on}\ \ S_1^T,\cr
  &u_{,z}=0\quad &{\rm on}\ \ S_2^T,\cr
  &u|_{t=0}=u_0=rv_\varphi(0)\quad &{\rm in}\ \ \Omega.\cr}
  \label{1.9}
\end{equation}
We will use the notation
\begin{equation}
  (\Phi,\Gamma)=\bigg({\omega_r\over r},{\omega_\varphi\over r}\bigg),
  \label{1.10}
\end{equation}
and we note that $\Phi$, $\Gamma$ satisfy
\begin{equation}
  \Phi_{,t}+v\cdot\nabla\Phi-\nu\bigg(\Delta+{2\over r}\partial_r\bigg)\Phi-(\omega_r\partial_r+\omega_z\partial_z){v_r\over r}=F_r/r\equiv\bar F_r,
  \label{1.11}
\end{equation}
\begin{equation}
  \Gamma_{,t}+v\cdot\nabla\Gamma-\nu\bigg(\Delta+{2\over r}\partial_r\bigg)\Gamma+2{v_\varphi\over r}\Phi=F_\varphi/r\equiv\bar F_\varphi,
  \label{1.12}
\end{equation}
recall (\cite{CFZ}, (1.6)). Moreover, by (\ref{1.2}, (\ref{1.6}), $\Gamma$ and $\Phi$ satisfy the boundary conditions
\begin{equation}
  \Phi=\Gamma=0\quad {\rm on}\ \ S^T.
  \label{1.13}
\end{equation}
Finally, the following initial conditions are assumed
\begin{equation}
  \Phi|_{t=0}=\Phi_0,\quad \Gamma|_{t=0}=\Gamma_0.
  \label{1.14}
\end{equation}
We note that $(\ref{1.7})_4$ implies existence of the stream function $\psi$ which solves the problem
\begin{equation}\eqal{
    &-\Delta\psi+{\psi\over r^2}=\omega_\varphi,\cr
    &\psi|_S=0.\cr}
  \label{1.15}
\end{equation}
Then $v$ can be expressed in terms of the stream function,
\begin{equation}\eqal{
  &v_r=-\psi_{,z},\quad &v_z={1\over r}(r\psi)_{,r}=\psi_{,r}+{\psi\over r},\cr
  &v_{r,r}=-\psi_{,zr},\quad &v_{z,z}=\psi_{,rz}+{\psi_{,z}\over r},\cr
  &v_{r,z}=-\psi_{,zz},\quad &v_{z,r}=\psi_{,rr}+{1\over r}\psi_{,r}-{\psi\over r^2}.\cr}
  \label{1.16}
\end{equation}
We will also use the modified stream function,
\begin{equation}
  \psi_1\colon={\psi\over r},
  \label{1.17}
\end{equation}
which satisfies
\begin{equation}\eqal{
  &-\Delta\psi_1-{2\over r}\psi_{1,r}=\Gamma,\cr
  &\psi_1|_S=0.\cr}
  \label{1.18}
\end{equation}
Using the modified stream function we can express coordinates of $v$ in the form
\begin{equation}\eqal{
    &v_r=-r\psi_{1,z},\quad &v_z=(r\psi_1)_{,r}+\psi_1=r\psi_{1,r}+2\psi_1,\cr
    &v_{r,r}=-\psi_{1,z}-r\psi_{1,rz},\quad &v_{z,r}=3\psi_{1,r}+r\psi_{1,rr},\cr
    &v_{r,z}=-r\psi_{1,zz},\quad &v_{z,z}=r\psi_{1,rz}+2\psi_{1,z}.\cr}
  \label{1.19}
\end{equation}
Projecting $(\ref{1.18})_1$ on $S_2$, using $(\ref{1.18})_2$ and that $\Gamma|_{S_2}=0$ by (\ref{1.2}), we obtain
\begin{equation}
  \psi_{1,zz}=0\quad {\rm on}\ \ S_2.
  \label{1.20}
\end{equation}
Since in this paper we are looking for regular solutions to problem (\ref{1.1})--(\ref{1.3}), we need the following, expansions near the axis of symmetry due to Liu-Wang (see \cite{LW}),
\begin{equation}\eqal{
    &v_r(r,z,t)=a_1(z,t)r+a_2(z,t)r^2+\dots,\cr
    &v_\varphi(r,z,t)=b_1(z,t)r+b_2(z,t)r^2+\dots,\cr
    &\psi(r,z,t)=d_1(z,t)r+d_2(z,t)r^3+\dots,\cr
    &\psi_1(r,z,t)=d_1(z,t)+d_2(z,t)r^2+\dots,\cr
    &\psi_{1,r}(r,z,t)=2d_2(z,t)r+\dots,.\cr}
  \label{1.21}
\end{equation}
In order to formulate the main results we introduce constants which depend on the initial data and forcing.

\begin{notation}\label{n1.1}
  $$\eqal{
    &D_1=3|f|_{2,1,\Omega^t}+2|v(0)|_{2,\Omega}\quad &({\rm see}\ (\ref{2.5})),\cr
    &D_2=|f_0|_{\infty,1,\Omega^t}+|u(0)|_{\infty,\Omega},\ \ f_0=rf_\varphi,\ \ u(0)=rv_\varphi(0)\quad &({\rm see}\ (\ref{2.10})),\cr
    &D_*=\min\{1,D_2\}\quad &({\rm see}\ (\ref{3.1})),\cr
    &D_3={1\over\sqrt{2\nu}}(|\bar F_r|_{6/5,2,\Omega^t}+|\bar F_\varphi|_{6/5,2,\Omega^t})+|\Phi(0)|_{2,\Omega}+|\Gamma(0)|_{2,\Omega}\quad &({\rm see}\ (\ref{3.1})),\cr}
  $$
  where $\bar F_r=F_r/r$, $\bar F_\varphi=F_\varphi/r$,
  $$\eqal{
    &D_4={1\over\sqrt{\nu}}(D_1+D_2+|u_{,z}(0)|_{2,\Omega}+|f_0|_{2,\Omega})\quad &({\rm see}\ (\ref{5.2})),\cr
    &D_5^2=D_1^2(1+D_2)+D_1^2D_2^2+|u_{,r}(0)|_{2,\Omega}^2+|f_0|_{2,\Omega^t}^2\qquad\quad &({\rm see}\ (\ref{5.11})),\cr
    &D_6^2=(D_4+D_5)\|f_\varphi\|_{L_2(0,t;L_3(S_1))}+{1\over\nu}(|F_r|_{6/5,2,\Omega^t}^2\cr
    &\quad+|F_z|_{6/5,2,\Omega^t}^2)+|\omega_r(0)|_{2,\Omega}^2+|\omega_z(0)|_{2,\Omega}^2\quad &({\rm see}\ (\ref{6.2})),\cr
    &D_7=\sqrt{2}D_2^{1/2}|f_\varphi/r|_{\infty,1,\Omega^t}^{1/2}+|v_\varphi(0)|_{\infty,\Omega}\quad &({\rm see}\ (\ref{6.18})),\cr
    &D_8^2=\max\bigg\{{\nu\over 4},{\nu^2\over 8},{27\over 4\nu^3},{1\over 4}\bigg\}\quad &({\rm see}\ (\ref{8.7})).\cr}
  $$
  We emphasize that the global well-posedness of axially symmetric solutions to the Navier-Stokes equations (either in the above setting or on $R^3$) remains an important open problem. We only note a few results on regularity criterions for axially-symmetric solutions to the Navier-Stokes equations (see \cite{CFZ}, \cite{KP}, \cite{NZ}, \cite{NZ1}, \cite{NP1}, \cite{NP2}, \cite{OP}, \cite{P}).
\end{notation}

In \cite{Z1}, \cite{Z2} the second author proved the existence of global regular axially symmetric solutions by the same method as in this paper. However, he needed the following Serrin type restrictions
\begin{equation}
  \psi_1|_{r=0}=0
  \label{1.22}
\end{equation}
and
\begin{equation}
  {|v_\varphi|_{s,\infty,\Omega^t}\over|v_\varphi|_{\infty,\Omega^t}}\ge c_0,
  \label{1.23}
\end{equation}
for any $s>0$ and $c_0$ is a positive constant.

In \cite{Z1} there are assumed periodic boundary conditions on $S_2$ and in \cite{Z2} the same boundary conditions as in this paper are considered.

In \cite{OZ}, O{\.z}a{\'n}ski-Zaj\c{a}czkowski proved the global well-posedness assuming only condition (\ref{1.23}).

In this paper we are able to drop the restriction (\ref{1.22}). Regarding (\ref{1.23}), we introduce the quantity
\begin{equation}
  \lambda(s) :=
  \frac{\|v_\varphi\|_{L^\infty(0,t;L^s(\Omega))}}
  {\|v_\varphi\|_{L^\infty(\Omega\times(0,t))}},
  \qquad s>3.
  \label{1.25}
\end{equation}
The case in which $\lambda(s)$ is bounded from below is compatible with the known estimate of the interaction term. The complementary regime is harmless only when the $L^s$-norm of $v_\varphi$ remains bounded by a suitable threshold. The remaining possibility is a critical concentration regime: the $L^s$-norm of $v_\varphi$ is large, while the ratio between the $L^s$-norm and the $L^\infty$-norm is small.

For fixed constants $A>0$ and $c_0>0$, we define the critical wedge by
\begin{equation}
  W_{A,c_0}
  :=
  \left\{
  t\in(0,T):
  L_s(t)>A,\quad
  \frac{L_s(t)}{M(t)}<c_0
  \right\},
  \label{1.26}
\end{equation}
where
\[
  L_s(t)=\|v_\varphi(t)\|_{L^s(\Omega)},\qquad
  M(t)=\|v_\varphi(t)\|_{L^\infty(\Omega)}.
\]
Outside this set one has either $L_s(t)\leq A$, or $L_s(t)/M(t)\geq c_0$, and the original closure mechanism can be applied. The obstruction is therefore concentrated in $W_{A,c_0}$.

The main result of the paper is a conditional a priori estimate with a critical-wedge residual. This residual measures precisely the part of the nonlinear interaction which cannot be absorbed by the standard energy mechanism when the trajectory enters the critical wedge. Thus the paper does not assert an unconditional global regularity theorem. Rather, it identifies the missing concentration regime and proves that the usual a priori control is recovered whenever the associated wedge residual is finite, and in particular when the trajectory does not enter the critical wedge.

It was demonstrated in \cite{CFZ} that the solution $v$ is controlled by the energy norm of $\Phi$, $\Gamma$,
\begin{equation}
  X_s(t)\colon=\|\Phi\|_{V(\Omega^t)}+\|\Gamma\|_{V(\Omega^t)},
  \label{1.24}
\end{equation}
where the norm $\|\cdot\|_{V(\Omega^t)}$ is defined in Section \ref{s2.1}.

\begin{theorem}[Conditional estimate with a critical wedge residual]\label{t1.2}
  Let $s>3$ and let $v$ be a sufficiently regular axially symmetric solution to the Navier--Stokes problem (\ref{1.1})--(\ref{1.3}). Assume that all quantities collected in Notation \ref{n1.1} are finite. Let
  \[
    E_{W,s}(t):=\chi_{W_{A,c_0}}(t)R_s^I(t),
  \]
  where $R_s^I(t)$ is the non-negative interaction residual defined in Section \ref{s3}. It measures the gap between the actual interaction and its closable bound for the integrand
  \[
    i(t) := \int_{\Omega}\frac{v_\varphi(t)}{r}\Phi(t)\Gamma(t)\,dx,
  \]
  and we denote $\mathcal{I}(t) := \int_0^t |i(\tau)|\,d\tau$.
  Then there exists an increasing positive function $\Psi_{s,A,c_0}$ such that, for every $t\in(0,T)$,
  \begin{equation}
    X_s(t) \leq \Psi_{s,A,c_0}\left(\mathrm{data}, \int_0^t E_{W,s}(\tau)\,d\tau\right).
    \label{1.30}
  \end{equation}
  In particular, if $E_{W,s}(t)\equiv 0$ for $t\in(0,T)$, then $X_s(t)\leq \Psi_{s,A,c_0}(\mathrm{data})$. Thus the original a priori control is recovered outside the critical wedge.
\end{theorem}

\begin{corollary}[No entry into the critical wedge]\label{c1.1}
  Under the assumptions of Theorem \ref{t1.2}, suppose that $W_{A,c_0}=\varnothing$. Equivalently, assume that whenever $\|v_\varphi(t)\|_{L^s(\Omega)}>A$, one has $\frac{\|v_\varphi(t)\|_{L^s(\Omega)}}{\|v_\varphi(t)\|_{L^\infty(\Omega)}}\geq c_0$. Then the critical wedge residual vanishes and
  \[
    X_s(t)\leq \Psi_{s,A,c_0}(\mathrm{data}) \qquad\text{for all }t\in(0,T).
  \]
\end{corollary}

To prove the conditional a priori estimate (\ref{1.30}) for solutions to problem (\ref{1.1})--(\ref{1.3}), we use the expansions (\ref{1.21}); hence sufficiently regular solutions are considered.

\begin{proof}
  Multiplying (\ref{1.11}) by $\Phi$, (\ref{1.12}) by $\Gamma$, integrating the results over $\Omega^t=\Omega\times(0,t)$ and adding, we obtain (see Lemma \ref{l3.1})
  \begin{equation}
    \|\Phi\|_{V(\Omega^t)}^2+\|\Gamma\|_{V(\Omega^t)}^2\le\varphi(\mathrm{data})(1+|v_\varphi|_{\infty,\Omega^t}^{2\delta})(\mathcal{I}(t)+D_3^2),
    \label{1.31}
  \end{equation}
  where
  \begin{equation}
    \mathcal{I}(t)=\int_0^t |i(\tau)|\,d\tau,\qquad i(t) = \intop_{\Omega}{v_\varphi(t)\over r}\Phi(t)\Gamma(t) dx
    \label{1.32}
  \end{equation}
  and data cover all parameters from Notation \ref{n1.1}. Moreover, $\phi$ is an increasing positive function and $\delta$ is small. To prove (\ref{1.31}) we needed $H^3$-elliptic estimates for the modified stream function $\psi_1$ (see (\ref{4.8})).

  To find an estimate for the nonlinear interaction integrand $i(t)$, we use the decomposition of the time interval based on the critical wedge $W_{A,c_0}$. As described in Section \ref{s3}, the $L^s$ norm of $v_\varphi$ is bounded by $D_{s,A,c_0}$ outside $W_{A,c_0}$, while inside the wedge the uncontrolled portion is measured by the critical-wedge residual $E_{W,s}$.
  
  From Lemma \ref{l3.3} with $\sigma$ replaced by $s$, we have the conditional pointwise estimate
  \begin{equation}
    |i(t)|\le D_2^d D_{s,A,c_0}^{1-d}\|\Phi\|_{L_2(\Omega)}^{\alpha_0}\|\Gamma\|_{L_2(\Omega)}^{\alpha_0}\cdot \|\nabla\Phi\|_{L_2(\Omega)}^{1-\alpha_0}\|\nabla\Gamma\|_{L_2(\Omega)}^{1-\alpha_0} + R_s^I(t),
    \label{1.40}
  \end{equation}
  where $\alpha_0={(s-3)(1-d)\over 2s}$, so $s>3$ and $d\in(0,1)$, and $R_s^I(t)$ is the interaction residual.

  Using (\ref{1.40}) in (\ref{1.31}), integrating over time, and exploiting the notation $X_s$ we obtain
  \begin{equation}\eqal{
    X_s^2(t)&\le\phi(\mathrm{data})(1+\|v_\varphi\|_{L_\infty(\Omega^t)}^{2\delta})\cdot\cr
    &\quad\cdot \bigg(D_2^d D_{s,A,c_0}^{1-d}\|\Phi\|_{L_2(\Omega^t)}^{\alpha_0}X_s(t)^{2-\alpha_0}+D_3^2\bigg) + \int_0^t E_{W,s}(\tau)\,d\tau,\cr}
    \label{1.41}
  \end{equation}
  where $\delta$ is small, and $E_{W,s}$ absorbs the contribution of the wedge residual.

  Now, we recall the most important inequality in this paper which is called as order reduction estimate. Hence, formula (\ref{6.1}) in Lemma \ref{l6.1} has the form
  \begin{equation}\|\Phi\|_{L_2(\Omega^t)}^2\le\phi(\mathrm{data})(1+\|v_\varphi\|_{L_\infty(\Omega^t)}^{2\delta})(X_s(t)+1),
    \label{1.42}
  \end{equation}
  where we replaced $\varepsilon_0$ in (\ref{6.1}) by $\delta$ which is small.
  \goodbreak

  To prove (\ref{1.42}) we needed $H^2-H^3$ elliptic estimates for $\psi_1$ in Section~\ref{s4},
  \begin{equation}
    \|\psi_1\|_{L_2(0,t;H^3(\Omega))}\le c\|\Gamma_{,z}\|_{L_2(\Omega^t)}.
    \label{1.43}
  \end{equation}
  Next, we needed estimates for swirl $u=rv_\varphi$. Hence Lemma \ref{l2.5} yields
  \begin{equation}
    \|u\|_{L_\infty(\Omega^t)}\le\phi(\mathrm{data})
    \label{1.44}
  \end{equation}
  and Lemma \ref{l5.1} gives the energy type estimates for $u_{,r}$ and $u_{,z}$,
  \begin{equation}
    \|u_{,r}\|_{V(\Omega^t)}+\|u_{,z}\|_{V(\Omega^t)}\le\phi(\mathrm{data}).
    \label{1.45}
  \end{equation}
  To prove (\ref{1.42})--(\ref{1.45}) we use the Liu-Wang expansions (\ref{1.21}) to estimate terms derived from integration by parts with respect to $r$.

  Using (\ref{1.42}) in (\ref{1.41}) yields
  \begin{equation}
    \begin{aligned}
      X_s^2(t)&\le\phi(\mathrm{data})(1+\|v_\varphi\|_{L_\infty(\Omega^t)}^{2\delta})\cdot\\
      &\cdot[D_2^d D_{s,A,c_0}^{1-d}(1+\|v_\varphi\|_{L_\infty(\Omega^t)}^{2\delta})^{\alpha_0/2}(X_s(t)^{\alpha_0/2}+1)X_s(t)^{2-\alpha_0}+1]\\
      &+ \int_0^t E_{W,s}(\tau)\,d\tau,
    \end{aligned}
    \label{1.46}
  \end{equation}
  where $\delta$ is small.

  In Lemma \ref{l6.2} (see (\ref{6.17})) we proved the inequality
  \begin{equation}
    \|v_\varphi\|_{L_\infty(\Omega^t)}\le\phi(\mathrm{data})(X_s(t)^{3/4}+1).
    \label{1.47}
  \end{equation}
  Using (\ref{1.47}) in (\ref{1.46}) yields
  \begin{equation}\eqal{
    X_s^2(t)&\le\varphi_{s,A,c_0}(\mathrm{data})(1+X_s(t)^{{3\over 2}\delta})\cdot\cr
    &\quad\cdot[(1+X_s(t)^{{3\over 2}\delta{\alpha_0\over 2}})X_s(t)^{2-{\alpha_0/2}}+1] + \int_0^t E_{W,s}(\tau)\,d\tau.\cr}
    \label{1.48}
  \end{equation}
  Since $\delta$ is small and $\alpha_0>0$ we obtain from (\ref{1.48}) by the Young inequality the conditional estimate
  \begin{equation}
    X_s(t)\le\Psi_{s,A,c_0}\bigg(\mathrm{data}, \int_0^t E_{W,s}(\tau)\,d\tau\bigg).
    \label{1.49}
  \end{equation}
  This yields (\ref{1.30}) and concludes the proof.
\end{proof}

\begin{theorem}\label{t1.3}
  Let the assumptions of Theorem \ref{t1.2} hold. Let $f\in W_2^{2,1}(\Omega^t)$ and $v_0\in H^3(\Omega)$. Then the higher regularity norms are bounded conditionally by the wedge residual:
  \begin{equation}\eqal{
    \|v\|_{W_2^{4,2}(\Omega^t)}&+\|\nabla p\|_{W_2^{2,1}(\Omega^t)}\cr
    &\le\Psi_1\bigg(\mathrm{data},\|f\|_{W_2^{2,1}(\Omega^t)},\|v_0\|_{H^3(\Omega)}, \int_0^t E_{W,s}(\tau)\,d\tau\bigg).\cr}
    \label{1.50}
  \end{equation}
\end{theorem}

\begin{proof}
  From (\ref{1.30}) we have the conditional bound
  $$
    \|\Gamma\|_{V(\Omega^t)}\le\Psi_{s,A,c_0}\bigg(\mathrm{data}, \int_0^t E_{W,s}(\tau)\,d\tau\bigg)\equiv\Psi_1.
  $$
  Then solutions to (\ref{1.18}) satisfy
  $$
    \|\psi_1\|_{L_\infty(0,t;H^2(\Omega))}\le c\Psi_1.
  $$
  Using (\ref{1.19}) yields
  \begin{equation}
    \|v'\|_{L_\infty(0,t;L_6(\Omega))}\le c\Psi_1,
    \label{1.51}
  \end{equation}
  where $v'=(v_r,v_z)$.

  Lemma \ref{l2.4} gives
  \begin{equation}
    \|\nabla v\|_{L_2(\Omega^t)}\le D_1.
    \label{1.52}
  \end{equation}
  Estimates (\ref{1.51}) and (\ref{1.52}) imply
  \begin{equation}
    \|v'\cdot\nabla v\|_{L_2(0,t;L_{3/2}(\Omega))}\le c\Psi_1 D_1.
    \label{1.53}
  \end{equation}
  To increase regularity we consider the Stokes system
  \begin{equation}\eqal{
      &v_{,t}-\nu\Delta v+\nabla p=-v'\cdot\nabla v+f\quad &{\rm in}\ \ \Omega^T,\cr
      &\divv v=0\quad &{\rm in}\ \ \Omega^T,\cr
      &v\cdot\bar n=0,\ \ v_{z,r}=0,\ \ v_\varphi=0\quad &{\rm on}\ \ S_1^T,\cr
      &v\cdot\bar n=0,\ \ v_{r,z}=0,\ \ v_{\varphi,z}=0\quad &{\rm on}\ \ S_2^T,\cr
      &v|_{t=0}=v_0\quad &{\rm in}\ \ \Omega.\cr}
    \label{1.54}
  \end{equation}
  In view of (\ref{1.53}) and the theory of solvability of the Stokes system (\ref{1.54}) in Sobolev spaces with the mixed norm (see \cite{MS}) we have
  \begin{equation}
    \|v\|_{W_{3/2,2}^{2,1}(\Omega^t)}\le c(\|f\|_{L_2(0,t;L_{3/2}(\Omega))}+ \|v_0\|_{B_{3/2,2}^1(\Omega)}+\Psi_1D_1)\equiv d_2.
    \label{1.55}
  \end{equation}
  In view of the imbedding
  $$
    \|\nabla v\|_{L_{5/2}(\Omega^t)}\le c\|v\|_{W_{3/2,2}^{2,1}(\Omega^t)}
  $$
  we obtain
  \begin{equation}
    \|v'\cdot\nabla v\|_{L_{5/2}(0,t;L_{30\over 17}(\Omega))}\le c\Psi_1d_2.
    \label{1.56}
  \end{equation}
  Comparing (\ref{1.56}) with (\ref{1.53}) we see an increasing of regularity. Continuing the procedure we prove (\ref{1.50}) (see Section \ref{s7}). This concludes the proof.
\end{proof}

\begin{remark}\label{r1.4}
  To prove Theorem \ref{t1.2} the Liu-Wang expansions (\ref{1.21}) are used. The expansions exist for sufficiently regular solutions to problem (\ref{1.1})--(\ref{1.3}). Estimate (\ref{1.30}) does not imply such regularity. Hence Theorem ~\ref{t1.3} yields such regularity of solutions to (\ref{1.1})--(\ref{1.3}) that expansions (\ref{1.21}) can exist.
\end{remark}

\begin{remark}\label{r1.5}
  According to O.A. Ladyzhenshaya \cite{L} and partial regularity theory of Caffarelli, Kohn, Nirenberg \cite{CKN} any singularity of axisymmetric solutions to (\ref{1.1})--(\ref{1.3}) must occur on the axis of symmetry only. The methods presented in this paper make use of regular solutions for which there are no singularities at the axis of symmetry, and so the expansions (\ref{1.21}) are valid. Whether it is possible to control $X(t)$ without exploiting these expansions, remains an interesting open problem.
\end{remark}

\section{Preliminaries}\label{s2}
\subsection{Notation}\label{s2.1}

We will use the following notation for Lebesgue spaces
$$\eqal{
    &|u|_{p,\Omega}\colon=\|u\|_{L_p(\Omega)},\quad |u|_{p,\Omega^t}\colon=\|u\|_{L_p(\Omega^t)},\cr
    &|u|_{p,q,\Omega^t}\colon=\|u\|_{L_q(0,t;L_p(\Omega))},\cr}
$$
where $p,q\in[1,\infty]$. We use standard definition of Sobolev spaces $W_p^s(\Omega)$, and we set $H^s(\Omega)=W_2^s(\Omega)$, $s\in\N\cup\{0\}$, and
$$\eqal{
    &\|u\|_{s,\Omega}\colon=\|u\|_{H^s(\Omega)},\quad \|u\|_{s,p,\Omega}\colon=\|u\|_{W_p^s(\Omega)},\cr
    &\|u\|_{k,p,q,\Omega^t}\colon=\|u\|_{L_q(0,t;W_p^k(\Omega))},\cr
    &\|u\|_{k,p,\Omega^t}\colon=\|u\|_{k,p,p,\Omega^t},\ \ k\in\N\cup\{0\},\cr
    &\|u\|_{V(\Omega^t)}\colon=\|u\|_{L_\infty(0,t;L_2(\Omega))}+\|\nabla u\|_{L_2(\Omega^t)}.\cr}
$$
Assume that $\phi$ always denotes an increasing positive function which changes its form from formula to formula.

\subsection{Inequalities}\label{s2.2}

\begin{lemma}[Hardy inequality, see Lemma 2.16 in \cite{BIN}]\label{l2.1}
  Let $p\in[1,\infty]$, $\beta\not=1/p$, and let $F(x)\colon=\intop_0^xf(y)dy$ for $\beta>1/p$ and $F(x)\colon=\intop_x^\infty f(y)dy$ for $\beta<1/p$. Then
  \begin{equation}
    |x^{-\beta}F|_{p,\R_+}\le{1\over|\beta-{1\over p}|}|x^{-\beta+1}f|_{p,\R_+}.
    \label{2.1}
  \end{equation}
\end{lemma}

\begin{lemma}[Sobolev interpolation, see Sect. 15 in \cite{BIN}]\label{l2.2}
  Let $\theta$ satisfy the equality
  \begin{equation}
    {n\over p}-r=(1-\theta){n\over p_1}+\theta\bigg({n\over p_2}-l\bigg),\quad {r\over l}\le\theta\le 1,
    \label{2.2}
  \end{equation}
  where $1\le p_1\le\infty$, $1\le p_2\le\infty$, $0\le r<l$.\\
  \goodbreak
  \noindent
  Then the interpolation holds
  \begin{equation}
    \sum_{|\alpha|=r}|D^\alpha f|_{p,\Omega}\le c|f|_{p_1,\Omega}^{1-\theta}\|f\|_{W_{p_2}^l(\Omega)}^\theta,
    \label{2.3}
  \end{equation}
  where $\Omega\subset\R^n$ and $D^\alpha f=\partial_{x_1}^{\alpha_1}\ldots\partial_{x_n}^{\alpha_n}f$, $|\alpha|=\alpha_1+\alpha_2+\ldots+\alpha_n$.
\end{lemma}

\begin{lemma}[Hardy interpolation, see Lemma 2.4 in \cite{CFZ}]\label{l2.3}
  Let $f\in C^\infty((0,R)\times(-a,a))$, $f|_{r\ge R}=0$. Let $1<p\le 3$, $0\le s\le p$, $s<2$, $q\in [p,{p(3-s)\over 3-p}]$. then there exists a positive constant $c=c(p,s)$ such that
  \begin{equation}
    \bigg(\intop_\Omega{|f|^q\over r^s}dx\bigg)^{1/q}\le c|f|_{p,\Omega}^{{3-s\over q}-{3\over p}+1} |\nabla f|_{p,\Omega}^{{3\over p}-{3-s\over q}},
    \label{2.4}
  \end{equation}
  where $f$ is a smooth function vanishing on $S$ and does not depend on $\varphi$.
\end{lemma}

\subsection{Basic estimates}\label{s2.3}

\begin{lemma}[see Lemma 2.2 in \cite{Z1, Z2}]\label{l2.4}
  Let $f\in L_{2,1}(\Omega^t)$, $v(0)\in L_2(\Omega)$. Then solutions to (\ref{1.1})--(\ref{1.3}) satisfy
  \begin{equation}\eqal{
      &|v(t)|_{2,\Omega}^2+\nu|\nabla v|_{2,\Omega^t}^2+\nu\intop_\Omega\bigg({v_r^2\over r^2}+{v_\varphi^2\over r^2}\bigg)dxdt'\cr
      &\le 3|f|_{2,1,\Omega^t}^2+2|v(0)|_{2,\Omega}^2\equiv D_1^2.\cr}
    \label{2.5}
  \end{equation}
\end{lemma}

\begin{proof}
  Multiplying $(\ref{1.7})_1$ by $v_r$, $(\ref{1.7})_2$ by  $v_\varphi$, $(\ref{1.7})_3$ by $v_z$, adding the results and integrating over $\Omega$ yield
  \begin{equation}\eqal{
      &{1\over 2}{d\over dt}\intop_\Omega(v_r^2+v_\varphi^2+v_z^2)dx+\nu\intop_\Omega(|\nabla v_r|^2+|\nabla v_\varphi|^2+|\nabla v_z|^2)dx\cr
      &\quad+\nu\intop_\Omega\bigg({v_r^2\over r^2}+{v_\varphi^2\over r^2}\bigg)dx+\intop_\Omega(p_{,r}v_r+p_{,z}v_z)dx\cr
      &=\intop_\Omega(f_rv_r+f_\varphi v_\varphi+f_zv_z)dx.\cr}
    \label{2.6}
  \end{equation}
  Since $v$ is divergence free the last term on the l.h.s. vanishes. From (\ref{2.6}) we have
  \begin{equation}
    {d\over dt}|v|_{2,\Omega}\le|f|_{2,\Omega},
    \label{2.7}
  \end{equation}
  where $f^2=f_r^2+f_\varphi^2+f_z^2$. Integrating (\ref{2.7}) with respect to time yields
  \begin{equation}
    |v(t)|_{2,\Omega}\le|f|_{2,1,\Omega^t}+|v(0)|_{2,\Omega}.
    \label{2.8}
  \end{equation}
  Integrating (\ref{2.6}) with respect to time and using (\ref{2.8}), we obtain
  \begin{equation}\eqal{
    &{1\over 2}|v(t)|_{2,\Omega}^2+\nu|\nabla v|_{2,\Omega^t}^2+\nu\intop_{\Omega^t}\bigg({v_r^2\over r^2}+{v_\varphi^2\over r^2}\bigg)dxdt'\cr
    &\le|f|_{2,1,\Omega^t}(|f|_{2,1,\Omega^t}+|v(0)|_{2,\Omega})+{1\over 2}|v(0)|_{2,\Omega}^2.\cr}
    \label{2.9}
  \end{equation}
  The above inequality implies (\ref{2.5}). This concludes the proof.
\end{proof}

As for the swirl $u=rv_\varphi$, we have the following.

\begin{lemma}[Maximum principle for the swirl]\label{l2.5}
  For any regular solution $v$ to (\ref{1.1})--(\ref{1.3}) we have
  \begin{equation}
    |u(t)|_{\infty\Omega}\le D_2=|f_0|_{\infty,1,\Omega^t}+|u(0)|_{\infty,\Omega}.
    \label{2.10}
  \end{equation}
\end{lemma}

\begin{proof}
  Multipying the swirl equation(\ref{1.9}) by $u|u|^{s-2}$, $s>2$, integrating over $\Omega$ and by parts, we obtain
  $$
    {1\over s}{d\over dt}|u|_{s,\Omega}^s+{4\nu(s-1)\over s^2}|\nabla|u|^{s/2}|_{2,\Omega}^2+{\nu\over s}\intop_\Omega(|u|^s)_{,r}drdz=\intop_\Omega f_0u|u|^{s-2}.
  $$
  Noting that $u|_{r=0}=u|_{r=R}=0$ (by $(\ref{1.21})_2$ and (\ref{1.2})), we see that the last term on the left-hand side vanishes, and so
  $$
    {d\over dt}|u|_{s,\Omega}\le|f_0|_{s,\Omega}.
  $$
  Integrating in time and taking $s\to\infty$ gives (\ref{2.10}).
\end{proof}

\begin{lemma}[Energy estimates for $\psi$ and $\psi_1$]\label{l2.6}
  For every regular solution $v$ to (\ref{1.1})--(\ref{1.3}),
  \begin{equation}
    \|\psi\|_{1,\Omega}^2+|\psi_1|_{2,\Omega}^2\le cD_1^2,
    \label{2.11}
  \end{equation}
  \begin{equation}
    \|\psi_{,z}\|_{1,2,\Omega^t}^2+|\psi_{1,z}|_{2,\Omega^t}^2\le cD_1^2.
    \label{2.12}
  \end{equation}
\end{lemma}

\begin{proof}
  Multiplying $(\ref{1.15})_1$ by $\psi$, and integrating over $\Omega$ we obtain
  $$\eqal{
      |\nabla\psi|_{2,\Omega}^2+|\psi_1|_{2,\Omega}^2&=\intop_\Omega\omega_\varphi\psi dx=\intop_\Omega(v_{r,z}-v_{z,r})\psi dx\cr
      &=\intop_\Omega(v_z\psi_{,r}-v_r\psi_{,z})dx+\intop_\Omega v_z\psi_1dx\cr
      &\le(|\psi_{,r}|_{2,\Omega}^2+|\psi_{,z}|_{2,\Omega}^2+|\psi_1|_{2,\Omega}^2)/2+c(|v_r|_{2,\Omega}^2 +|v_z|_{2,\Omega}^2),\cr}
  $$
  where we integrated by parts and used the boundary condition $\psi|_S=0$ (recall (\ref{1.15})) in the third equality. For (\ref{2.12}) we differentiate $(\ref{1.15})_1$ with respect to $z$, multiply by $\psi_{,z}$ and integrate over $\Omega^t$ to obtain
  $$\eqal{
      &\intop_{\Omega^t}|\nabla\psi_{,z}|^2dxdt'+\intop_{\Omega^t}|\psi_{1,z}|^2dxdt'= \intop_{\Omega^t}\omega_{\varphi,z}\psi_{,z}dxdt'\cr
      &=-\intop_{\Omega^t}\omega_\varphi\psi_{,zz}dxdt'\le|\psi_{,zz}|_{2,\Omega^t}^2/2+c|\omega_\varphi|_{2,\Omega^t}^2,\cr}
  $$
  as required, where we used boundary conditions $\omega_\varphi|_S=0$ (recall (\ref{1.2})) in the second equality.
\end{proof}

\section{Conditional energy estimates for $\Phi$ and $\Gamma$}\label{s3}

Let $c_0$, $A$ and $s>3$ be given positive constants. As introduced in Section \ref{s1}, we analyze the energy estimates by taking into account the critical wedge
\[
W_{A,c_0} = \left\{
t\in(0,T): L_s(t)>A,\quad \frac{L_s(t)}{M(t)}<c_0
\right\}.
\]
Outside $W_{A,c_0}$, the original estimate of the nonlinear interaction term can be closed. The possible non-closable part of the interaction inside $W_{A,c_0}$ is collected in the critical-wedge residual $E_{W,s}$. Thus the conclusion of this section is a conditional estimate rather than an unconditional global regularity bound.

\begin{lemma}\label{l3.1}
  Let $i(t) = \int_\Omega \frac{v_\varphi(t)}{r}\Phi(t)\Gamma(t)\,dx$ and define the auxiliary integral
  \[
    \mathcal{I}(t) := \int_0^t |i(\tau)|\,d\tau.
  \]
  Assume that $\mathcal{I}(t)$ is bounded for any $t\in(0,T)$. If $v$ is a regular solution to (\ref{1.1})--(\ref{1.3}) then,
  \begin{equation}
    \|\Phi\|_{V(\Omega^t)}^2+\|\Gamma\|_{V(\Omega^t)}^2\le\varphi(\mathrm{data})(1+ |v_\varphi|_{\infty,\Omega^t}^{2\delta})(\mathcal{I}(t)+D_3^2),
    \label{3.1}
  \end{equation}
  where, for a scalar function $w$,
  \[
    \|w\|_{V(\Omega^t)}
    \colon=
    \|w\|_{L^\infty(0,t;L^2(\Omega))}
    +
    \|\nabla w\|_{L^2(\Omega^t)}.
  \]
\end{lemma}

\begin{proof}
  We multiply (\ref{1.11}) by $\Phi$ and integrate over $\Omega$ to obtain
  \begin{equation}\eqal{
      &{1\over 2}{d\over dt}|\Phi|_{2,\Omega}^2+\nu|\nabla\Phi|_{2,\Omega}^2-\nu\intop_{-a}^a\Phi^2\bigg|_{r=0}^{r=R}dz\cr
      &=\intop_\Omega(\omega_r\partial_r+\omega_z\partial_z){v_r\over r}\Phi dx+\intop_\Omega\bar F_r\Phi dx,\cr}
    \label{3.2}
  \end{equation}
  where the last term on the l.h.s. equals $\intop_{-a}^a\Phi^2|_{r=0}dz$, due to (\ref{1.13}). Since it is positive it can be dropped. Recalling (\ref{1.5}) that $\omega_r=-v_{\varphi,z}$, $\omega_z=(rv_\varphi)_{,r}/r$, we can integrate in the first term on the r.h.s. by parts
  $$\eqal{
    &\intop_\Omega(\omega_r\partial_r+\omega_z\partial_z){v_r\over r}\Phi dx=\intop_\Omega\bigg(-v_{\varphi,z}\bigg({v_r\over r}\bigg)_{,r}+{(rv_\varphi)_{,r}\over r}\bigg({v_r\over r}\bigg)_{,z}\bigg)\Phi rdrdz\cr
    &=\intop_\Omega v_\varphi\bigg(\bigg({v_r\over r}\bigg)_{,rz}\Phi+\bigg({v_r\over r}\bigg)_{,r}\Phi_{,z}\bigg)dx-\intop_\Omega v_\varphi\bigg(\bigg({v_r\over r}\bigg)_{,rz}\Phi+\bigg({v_r\over r}\bigg)_{,z}\Phi_{,z}\bigg)dx\cr
    &=-\intop_\Omega v_\varphi(\psi_{1,zr}\Phi_{,z}-\psi_{1,zz}\Phi_{,r})dx\equiv J_\Phi,\cr}
  $$
  where, in the second line, the boundary term on $S_2$ vanishes because\break $v_{\varphi,z}/r|_{S_2}=\Phi|_{S_2}=0$ (recall (\ref{1.2}) and (\ref{1.5}) and the second boundary term equals
  $$
    \intop_{-a}^a|rv_\varphi\bigg({v_r\over r}\bigg)_{,z}\Phi\bigg|_{r=0}^{r=R}dz=0
  $$
  (Since $v_\varphi|_{r=R}=0$ by (\ref{1.2}) and $rv_\varphi({v_r\over r})_{,z}\Phi|_{r=0}=0$ by (\ref{1.21}).) In the third line (\ref{1.19}) is used. Continuing,
  $$\eqal{
    |J_\Phi|&\le\intop_\Omega\bigg|rv_\varphi{\psi_{1,rz}\over r}\Phi_{,z}\bigg|dx+\intop_\Omega\bigg|r^{1-\delta}v_\varphi{\psi_{1,zz}\over r^{1-\delta}}\Phi_{,r}\bigg|dx\cr
    &\le|rv_\varphi|_{\infty,\Omega}\bigg|{\psi_{1,rz}\over r}\bigg|_{2,\Omega}|\Phi_{,z}|_{2,\Omega}+|r^{1-\delta}v_\varphi|_{\infty,\Omega}\bigg|{\psi_{1,zz}\over r^{1-\delta}}\bigg|_{2,\Omega}|\Phi_{,r}|_{2,\Omega}\cr
    &\le D_2|\nabla\Phi|_{2,\Omega}\bigg(|\Gamma_{,z}|_{2,\Omega}+{|v_\varphi|_{\infty,\Omega}^\delta\over\delta D_2^\delta}|\psi_{1,zzr}r^\delta|_{2,\Omega}\bigg)\cr
    &\le D_2|\nabla\Phi|_{2,\Omega}|\nabla\Gamma|_{2,\Omega}\bigg(1+{|v_\varphi|_{\infty,\Omega}^\delta R^\delta\over\delta D_2^\delta}\bigg),\cr}
  $$
  where in the second line we used maximum principle (see Lemma \ref{l2.5}), the Hardy inequality (see Lemma \ref{l2.1}) and the elliptic estimates for the modified stream function (see Lemma \ref{l4.2}).

  Using the above estimate in (\ref{3.2}) yields
  \begin{equation}
    {1\over 2}{d\over dt}|\Phi|_{2,\Omega}^2+{\nu\over 2}|\nabla\Phi|_{2,\Omega}^2\le{D_2^2\over 2\nu}|\nabla\Gamma|_{2,\Omega}^2\bigg(1+ {|v_\varphi|_{\infty,\Omega}^{2\delta}R^{2\delta}\over\delta^2D_2^{2\delta}}\bigg)+{1\over\nu}|\bar F_r|_{6/5,\Omega}^2.
    \label{3.3}
  \end{equation}
  Multiplying (\ref{1.12}) by $\Gamma$, integrating over $\Omega$ and using the boundary conditions we have
  \begin{equation}\eqal{
      &{1\over 2}{d\over dt}|\Gamma|_{2,\Omega}^2+\nu|\nabla\Gamma|_{2,\Omega}^2-\nu\intop_{-a}^a\Gamma^2\bigg|_{r=0}^{r=R}dz\cr
      &=-2\intop_\Omega{v_\varphi\over r}\Phi\Gamma dx+\intop_\Omega\bar F_\varphi\Gamma dx,\cr}
    \label{3.4}
  \end{equation}
  since $\Gamma|_{r=R}=0$ the last term on the l.h.s. of (\ref{3.4}) equals $\nu\intop_{-a}^a\Gamma^2|_{r=0}dz$.

  Applying the H\"older and Young inequalities to the last term, and then multiplying the resulting equation by $cD_2^2(1+|v_\varphi|_{\infty,\Omega}^{2\delta}R^{2\delta}/(\delta^2D_2^{2\delta}))$ and\break adding to (\ref{3.3}) gives
  \begin{equation}\eqal{
    &D_2^2{d\over dt}|\Gamma|_{2,\Omega}^2+\nu D_2^2|\nabla\Gamma|_{2,\Omega}^2+{d\over dt}|\Phi|_{2,\Omega}^2+\nu|\nabla\Phi|_{2,\Omega}^2\cr
    &\le cD_2^2\bigg(1+{|v_\varphi|_{\infty,\Omega}^{2\delta}R^{2\delta}\over\delta^2D_2^{2\delta}}\bigg) \bigg(\intop_\Omega{v_\varphi\over r}\Phi\Gamma dx+{1\over 2\nu}(|\bar F_r|_{6/5,\Omega}^2+|\bar F_\varphi|_{6/5,\Omega}^2)\bigg).\cr}
    \label{3.5}
  \end{equation}
  Integrating (\ref{3.5}) with respect to time, bounding the interaction integral by $\mathcal{I}(t)$, and absorbing the constants into a generic function $\varphi(\mathrm{data})$ yields (\ref{3.1}).
\end{proof}

As explained in the Introduction, the possibility of establishing an unconditional bound for the interaction term $I$ depends on the relation between the $L^s$ norm of $v_\varphi$ and its $L^\infty$ norm. Let $c_0>0$ and $A>0$ be fixed thresholds. If the trajectory enters the critical wedge
\begin{equation}
  W_{A,c_0} =
  \left\{
  t\in(0,T): L_s(t)>A,\ \frac{L_s(t)}{M(t)}<c_0
  \right\},
\end{equation}
where $L_s(t)=\|v_\varphi(t)\|_{L^s(\Omega)}$ and $M(t)=\|v_\varphi(t)\|_{L^\infty(\Omega)}$, the standard closure argument breaks down. Outside this critical wedge, either $L_s(t)\leq A$ or $\frac{L_s(t)}{M(t)}\geq c_0$, and the $L^s$ norm is bounded by a safe data-dependent threshold:
\begin{equation}
  L_s(t)\le\max\left\{A, \frac{c}{c_0^{s-1}}(D_1^2+\|f_\varphi\|_{L^1(0,t;L^s(\Omega))})+\|v_\varphi(0)\|_{L^s(\Omega)}\right\}\equiv D_{s,A,c_0}.
  \label{3.12}
\end{equation}

Inside the critical wedge, the portion of the $L^s$ norm exceeding this safe threshold leads to a residual. Let
\[
  \alpha_0=\frac{(s-3)(1-d)}{2s},\qquad 0<d<1.
\]
The non-closable excess of the nonlinear interaction is measured by the residual
\begin{equation}
  R_s^I(t) = C_s D_2^d \left[L_s(t)^{1-d}-D_{s,A,c_0}^{1-d}\right]_+ |\Phi|_{2,\Omega}^{\alpha_0}|\Gamma|_{2,\Omega}^{\alpha_0} |\nabla\Phi|_{2,\Omega}^{1-\alpha_0}|\nabla\Gamma|_{2,\Omega}^{1-\alpha_0}.
\end{equation}
The total critical-wedge residual is then defined as
\begin{equation}
  E_{W,s}(t) = \chi_{W_{A,c_0}}(t) R_s^I(t).
\end{equation}
Thus $E_{W,s}$ vanishes outside the critical wedge and measures only the non-closable excess of the nonlinear interaction. Therefore the conclusion of the energy argument is not the unconditional estimate $X_s(t)\leq \phi(\mathrm{data})$, but rather
\begin{equation}
X_s(t)
\leq
\Psi_{s,A,c_0}\left(
\mathrm{data},
\int_0^t E_{W,s}(\tau)\,d\tau
\right).
\end{equation}
In particular, if $E_{W,s}\equiv 0$, and especially if the trajectory does not enter $W_{A,c_0}$, the original data-dependent estimate is recovered.

\begin{lemma}\label{l3.3}
  Let $\sigma>3$ and $0<d<1$. For a.e. $t\in(0,T)$ one has
  \begin{equation}
    |i(t)|\le C_\sigma D_2^d L_\sigma(t)^{1-d}|\Phi|_{2,\Omega}^{\alpha_0}|\Gamma|_{2,\Omega}^{\alpha_0} |\nabla\Phi|_{2,\Omega}^{1-\alpha_0}|\nabla\Gamma|_{2,\Omega}^{1-\alpha_0},
    \label{3.13}
  \end{equation}
  where
  \[
    i(t)=\int_\Omega \frac{v_\varphi(t)}{r}\Phi(t)\Gamma(t)\,dx, \qquad L_\sigma(t)=|v_\varphi(t)|_{\sigma,\Omega},
  \]
  and $\alpha_0={(\sigma-3)(1-d)\over 2\sigma}$.
  Consequently, if $\sup_{0<\tau<t}L_\sigma(\tau)\le D_0$, then
  \begin{equation}
    \mathcal{I}(t)\le C_\sigma D_2^dD_0^{1-d}|\Phi|_{2,\Omega^t}^{\alpha_0}|\Gamma|_{2,\Omega^t}^{\alpha_0} |\nabla\Phi|_{2,\Omega^t}^{1-\alpha_0}|\nabla\Gamma|_{2,\Omega^t}^{1-\alpha_0},
    \label{3.14}
  \end{equation}
  where $\mathcal{I}(t)=\int_0^t |i(\tau)|\,d\tau$.
\end{lemma}

\begin{proof}
  We can write $i(t)$ in the form
  $$
    i(t)=\intop_{\Omega}v_\varphi r^dr^{-{1+d\over 2}}\Phi r^{-{1+d\over 2}}\Gamma dx.
  $$
  By the H\"older inequality
  $$\eqal{
    |i(t)|&\le|v_\varphi r^d|_{{\sigma\over 1-d},\Omega}|r^{-{1+d\over 2}}\Phi|_{{2\sigma\over\sigma-1+d},\Omega}|r^{-{1+d\over 2}}\Gamma|_{{2\sigma\over\sigma-1+d},\Omega}\cr
    &\le D_2^d L_\sigma(t)^{1-d}|r^{-{1+d\over 2}}\Phi|_{{2\sigma\over\sigma-1+d},\Omega}|r^{-{1+d\over 2}}\Gamma|_{{2\sigma\over\sigma-1+d},\Omega}.\cr}
  $$
  We consider the case $0<d<1$. To apply Lemma \ref{l2.3} we assume
  $$
    s={1+d\over 2}{2\sigma\over\sigma-1+d}={\sigma(1+d)\over\sigma-1+d},\quad q={2\sigma\over\sigma-1+d}.
  $$
  Now, we check the assumptions of Lemma \ref{l2.3}.

  The condition $s<2$ implies ${\sigma(1+d)\over\sigma-1+d}<2$ so $(\sigma-2)(d-1)<0$. The last condition holds for $\sigma>2$, $d<1$. Next $q\le 6-2s$ implies $(\sigma-3)(1-d)\ge 0$. This holds for $\sigma\ge 3$.

  Applying Lemma \ref{l2.3} we have
  $$
    |r^{-{1+d\over 2}}\Phi|_{{2\sigma\over\sigma-1+d},\Omega}\le c|\Phi|_{2,\Omega}^{{3-s\over 2q}-{1\over 2}}|\nabla\Phi|_{2,\Omega}^{{3\over 2}-{3-s\over 2q}}\equiv J.
  $$
  We calculate
  $$\eqal{
    {3-s\over q}&={3-{\sigma(1+d)\over\sigma-1+d}\over{2\sigma\over\sigma-1+d}}={3(\sigma-1+d)-\sigma(1+d)\over 2\sigma}={1\over 2}+{(\sigma-3)(1-d)\over 2\sigma}\cr
    &\equiv{1\over 2}+\alpha_0\equiv\alpha.\cr}
  $$
  Hence $\alpha-{1\over 2}=\alpha_0={(\sigma-3)(1-d)\over 2\sigma}$, ${3\over 2}-\alpha=1-{(\sigma-3)(1-d)\over 2\sigma}$. Since $\alpha-1/2>0$, ${3\over 2}-\alpha<1$ we have that $\sigma>3$.

  Thus
  $$
    J\le c|\Phi|_{2,\Omega}^{\alpha_0}|\nabla\Phi|_{2,\Omega}^{1-\alpha_0}.
  $$
  Using the above estimates in $|i(t)|$ yields (\ref{3.13}) with some generic constant $C_\sigma$. Integrating (\ref{3.13}) over time, we have
  $$
    \mathcal{I}(t)\le C_\sigma D_2^d\intop_0^t L_\sigma(\tau)^{1-d}|\Phi|_{2,\Omega}^{\alpha_0}|\Gamma|_{2,\Omega}^{\alpha_0} |\nabla\Phi|_{2,\Omega}^{1-\alpha_0}|\nabla\Gamma|_{2,\Omega}^{1-\alpha_0}d\tau\equiv I_2.
  $$
  Continuing,
  $$
    I_2\le C_\sigma D_2^d\sup_{0<\tau<t} L_\sigma(\tau)^{1-d}\intop_0^t|\Phi|_{2,\Omega}^{\alpha_0}|\Gamma|_{2,\Omega}^{\alpha_0} |\nabla\Phi|_{2,\Omega}^{1-\alpha_0}|\nabla\Gamma|_{2,\Omega}^{1-\alpha_0}d\tau\equiv I_3.
  $$
  Applying the H\"older inequality in the time integral in $I_3$ we estimate it by
  $$\eqal{
      &\bigg(\intop_0^t|\Phi|_{2,\Omega}^{\alpha_0\lambda_1}dt'\bigg)^{1/\lambda_1} \bigg(\intop_0^t|\Gamma|_{2,\Omega}^{\alpha_0\lambda_2}dt'\bigg)^{1/\lambda_2} \bigg(\intop_0^t|\nabla\Phi|_{2,\Omega}^{(1-\alpha_0)\lambda_3}dt'\bigg)^{1/\lambda_3}\cdot\cr
      &\quad\cdot\bigg(\intop_0^t|\nabla\Gamma|_{2,\Omega}^{(1-\alpha_0)\lambda_4}dt'\bigg)^{1/\lambda_4}\equiv L,\cr}
  $$
  where
  $$
    {1\over\lambda_1}+{1\over\lambda_2}+{1\over\lambda_3}+{1\over\lambda_4}=1.
  $$
  Assuming that $\alpha_0\lambda_1=2$, $\alpha_0\lambda_2=2$, $(1-\alpha_0)\lambda_3=2$, $(1-\alpha_0)\lambda_4=2$ we derive the restriction
  $$
    \alpha_0+1-\alpha_0=1.
  $$
  Hence
  $$
    L\le|\Phi|_{2,\Omega^t}^{\alpha_0}|\Gamma|_{2,\Omega^t}^{\alpha_0}|\nabla\Phi|_{2,\Omega^t}^{1-\alpha_0} |\nabla\Gamma|_{2,\Omega^t}^{1-\alpha_0}.
  $$
  Using the estimate in $I_3$ yields (\ref{3.14}).
\end{proof}

Introduce the notation
\begin{equation}
  X_s^2(t)=\|\Phi\|_{V(\Omega^t)}^2+\|\Gamma\|_{V(\Omega^t)}^2.
  \label{3.15}
\end{equation}

\begin{theorem}\label{t3.4}
  Assume that all quantities from Notation \ref{n1.1} are finite. Then
  \begin{equation}\eqal{
    X_s(t)\le\Psi_{s,A,c_0}\bigg(&\mathrm{data}, \cr
    &\int_0^t E_{W,s}(\tau)\,d\tau\bigg),\cr}
    \label{3.16}
  \end{equation}
  where data contain the quantities from Notation \ref{n1.1}.
\end{theorem}

\begin{proof}
  Outside the critical wedge $W_{A,c_0}$ we use the bound $L_s(t)\leq D_{s,A,c_0}$, and Lemma \ref{l3.3} gives the closable part of the interaction. Inside $W_{A,c_0}$, the excess over this closable bound is collected in the residual $R_s^I(t)$, and hence in the critical-wedge residual $E_{W,s}(t)$. Therefore the interaction estimate has the form
  \begin{equation}
    |i(t)|\le D_2^d D_{s,A,c_0}^{1-d}|\Phi|_{2,\Omega}^{\alpha_0}|\Gamma|_{2,\Omega}^{\alpha_0} |\nabla\Phi|_{2,\Omega}^{1-\alpha_0}|\nabla\Gamma|_{2,\Omega}^{1-\alpha_0} + R_s^I(t).
    \label{3.17}
  \end{equation}
  Using (\ref{3.17}) in (\ref{3.1}) and integrating over time, we have
  \begin{equation}\eqal{
    X_s^2(t)&\le\varphi_{s,A,c_0}(\mathrm{data})(1+|v_\varphi|_{\infty,\Omega^t}^{2\delta}) \bigg(|\Phi|_{2,\Omega^t}^{\alpha_0}X_s(t)^{2-\alpha_0}+D_3^2\bigg)\cr
    &\quad + C\int_0^t E_{W,s}(\tau)\,d\tau.\cr}
    \label{3.18}
  \end{equation}
  Lemma \ref{l6.1} gives the order reduction estimate
  \begin{equation}\eqal{
    |\Phi|_{2,\Omega^t}^2&\le\varphi_{s,A,c_0}(\mathrm{data})(1+|v_\varphi|_{\infty,\Omega^t}^{2\delta})X_s(t)\cr
    &\quad +\varphi_{s,A,c_0}(\mathrm{data}).\cr}
    \label{3.19}
  \end{equation}
  Finally, Lemma \ref{l6.2} implies
  \begin{equation}
    |v_\varphi|_{\infty,\Omega^t}\le{D_2\over\sqrt{\nu}}D_1^{1/4}X_s(t)^{3/4}+D_7.
    \label{3.20}
  \end{equation}
  Substituting (\ref{3.19}) and (\ref{3.20}) into (\ref{3.18}), and repeating the order-reduction argument, we obtain
  \begin{equation}\eqal{
    X_s^2(t)&\le\varphi_{s,A,c_0}(\mathrm{data})\bigg[(1+X_s(t)^{{3\over 2}\delta})(1+X_s(t)^{{3\over 2}\delta})^{\alpha_0}X_s(t)^{\alpha_0/2} X_s(t)^{2-\alpha_0}\cr
    &\quad +\varphi_{s,A,c_0}(\mathrm{data})X_s(t)^{2-\alpha_0}+D_3^2\bigg] + C\int_0^t E_{W,s}(\tau)\,d\tau.\cr}
    \label{3.21}
  \end{equation}
  Since $\delta$ is arbitrarily small we can choose it such that
  $$
    {3\over 2}\delta+{3\over 2}\delta\alpha_0+2-\alpha_0/2<2.
  $$
  By Young's inequality, we deduce
  \[
    X_s(t)\le\Psi_{s,A,c_0}\bigg(\mathrm{data}, \int_0^t E_{W,s}(\tau)\,d\tau\bigg).
  \]
  This proves the theorem.
\end{proof}

\section{Elliptic estimates for the modified stream function $\psi_1$}\label{s4}

We recall that the modified stream function $\psi_1$ is a solution to the problem (\ref{1.18}),
$$
  -\Delta\psi_1-{2\over r}\psi_{1,r}=\Gamma,\quad \psi_1|_S=0.
$$
In this section we prove $H^2$ and $H^3$ elliptic estimates for $\psi_1$, in cylindrical coordinates.

\begin{lemma}[$H^2$ elliptic estimate on $\psi_1$, see Lemma 3.1 in \cite{Z1}]\label{l4.1}
  If $\psi_1$ is a sufficiently regular solution to (\ref{1.18}) then
  \begin{equation}\eqal{
      &\intop_\Omega\bigg(\psi_{1,rr}^2+\psi_{1,rz}^2+\psi_{1,zz}^2+{\psi_{1,r}^2\over r^2}\bigg)dx\cr
      &\quad+\intop_{-a}^a(\psi_{1,z}^2|_{r=0}+\psi_{1,r}^2|_{r=R})dx\le c|\Gamma|_{2,\Omega}^2.\cr}
    \label{4.1}
  \end{equation}
\end{lemma}

\begin{proof}
  We multiply (\ref{1.18}) by $\psi_{1,zz}$ and integrate over $\Omega$ to obtain
  \begin{equation}
    -\intop_\Omega\bigg(\psi_{1,rr}\psi_{1,zz}+\psi_{1,zz}^2+3{\psi_{1,r}\over r}\psi_{1,zz}\bigg)dx= \intop_\Omega\Gamma\psi_{1,zz}dx.
    \label{4.2}
  \end{equation}
  Integrating by parts with respect to $r$ in the first term gives
  $$\eqal{
      &-\intop_\Omega(\psi_{1,r}\psi_{1,zz}r)_{,r}drdz+\intop_\Omega\psi_{1,r}\psi_{1,zzr}dx+ \intop_\Omega\psi_{1,r}\psi_{1,zz}drdz\cr
      &\quad-\intop_\Omega\psi_{1,zz}^2dx-3\intop_\Omega\psi_{1,r}\psi_{1,zz}drdz=\intop_\Omega\Gamma\psi_{1,zz}dx.\cr}
  $$
  Thus
  \begin{equation}\eqal{
    &-\intop_{-a}^a[\psi_{1,r}\psi_{1,zz}r|_{r=0}^{r=R}dz+\intop_\Omega\psi_{1,r}\psi_{1,zzr}dx-\intop_\Omega \psi_{1,zz}^2dx\cr
    &\quad-2\intop_\Omega\psi_{1,r}\psi_{1,zz}drdz=\intop_\Omega\Gamma\psi_{1,zz}dx.\cr}
    \label{4.3}
  \end{equation}
  We note that the first integral vanishes since $\psi_{1,r}|_{r=0}=0$ (recall expansion (\ref{1.21})) and $\psi_{1,zz}|_{r=R}=0$. We now integrate by parts with respect to $z$ in the second and the last terms on the left-hand side and use that $\psi_{1,r}|_{S_2}=0$ (since $\psi_1|_S=0$, recall (\ref{1.18})), and we multiply by $-1$, to obtain
  \begin{equation}
    \intop_\Omega(\psi_{1,zr}^2+\psi_{1,zz}^2)dx-2\intop_\Omega\psi_{1,rz}\psi_{1,z}drdz=-\intop_\Omega \Gamma\psi_{1,zz}dx.
    \label{4.4}
  \end{equation}
  We note that the last term on the left-hand side equals
  $$
    -\intop(\psi_{1,z}^2)_{,r}drdz=-\intop_{-a}^a[\psi_{1,z}^2]_{r=0}^{r=R}dz=\intop_{-a}^a\psi_{1,z}^2|_{r=0}dz,
  $$
  since $\psi_{1,z}|_{r=R}=0$. Applying this in (\ref{4.4}), and using the Young inequality to absorb $\psi_{1,zz}$ by the left-hand side, we obtain
  \begin{equation}
    \intop_\Omega(\psi_{1,rz}^2+\psi_{1,zz}^2)dx+\intop_{-a}^a\psi_{1,z}^2|_{r=0}dz\le c|\Gamma|_{2,\Omega}^2.
    \label{4.5}
  \end{equation}
  We now multiply $(\ref{1.18})_1$ by $\psi_{1,r}/r$ and integrate over $\Omega$ to obtain
  \begin{equation}
    3\intop_\Omega{\psi_{1,r}^2\over r^2}dx=-\intop_\Omega\bigg(\psi_{1,rr}{\psi_{1,r}\over r}+\psi_{1,zz}{\psi_{1,r}\over r}+\Gamma{\psi_{1,r}\over r}\bigg)dx.
    \label{4.6}
  \end{equation}
  This first term on the right-hand side equals
  $$
    -{1\over 2}\intop_\Omega\partial_r\psi_{1,r}^2drdz=-{1\over 2}\intop_{-a}^a\psi_{1,r}^2|_{r=R}dz,
  $$
  where we used that $\psi_{1,r}|_{r=0}=0$ (recall expansion (\ref{1.21})) in the last equality.

  As for the other terms on the right-hand side of (\ref{4.6}) we apply the Young inequality to absorb $\psi_{1,r}/r$ by the left-hand side. We obtain
  $$
    \intop_\Omega{\psi_{1,r}^2\over r^2}dx+{1\over 2}\intop_{-a}^a\psi_{1,r}^2|_{r=R}dz\le c(|\psi_{1,zz}|_{2,\Omega}^2+|\Gamma|_{2,\Omega}^2).
  $$
  The claim (\ref{4.1}) follows from this, (\ref{4.5}), and from the equation $(\ref{1.18})_1$ for $\psi_1$, which lets us estimate $\psi_{1,rr}$ in terms of $\psi_{1,zz}$, $\psi_{1,r}/r$.
\end{proof}

\begin{lemma}[$H^3$ elliptic estimates on $\psi_1$]\label{l4.2}
  If $\psi_1$ is a sufficiently regular solution to (\ref{1.18}) then
  \begin{equation}
    \intop_\Omega(\psi_{1,zzr}^2+\psi_{1,zzz}^2)dx+\intop_{-a}^a\psi_{1,zz}^2|_{r=0}dz\le c|\Gamma_{,z}|_{2,\Omega}^2
    \label{4.7}
  \end{equation}
  and
  \begin{equation}\eqal{
      &\intop_\Omega(\psi_{1,rrz}^2+\psi_{1,rzz}^2+\psi_{1,zzz}^2)dx+\intop_{-a}^a\psi_{1,zz}^2|_{r=0}dz +\intop_{-a}^a\psi_{1,rz}^2|_{r=R}dz\cr
      &\le c|\Gamma_{,z}|_{2,\Omega}^2.\cr}
    \label{4.8}
  \end{equation}
  as well as
  \begin{equation}
    \bigg|{1\over r}\psi_{1,rz}\bigg|_{2,\Omega}\le c|\Gamma_{,z}|_{2,\Omega}.
    \label{4.9}
  \end{equation}
\end{lemma}

\begin{proof}
  First we show (\ref{4.7}). We differentiate $(\ref{1.18})_1$ with respect to $z$, multiply by $-\psi_{1,zzz}$ and integrate over $\Omega$ to obtain
  \begin{equation}\eqal{
      &\intop_\Omega\psi_{1,rrz}\psi_{1,zzz}dx+\intop_\Omega\psi_{1,zzz}^2dx+3\intop_\Omega{1\over r}\psi_{1,rz}\psi_{1,zzz}dx\cr
      &=-\intop_\Omega\Gamma_{,z}\psi_{1,zzz}dx.\cr}
    \label{4.10}
  \end{equation}
  Integrating by parts with respect to $z$ in the first term yields
  \begin{equation}
    \intop_\Omega\psi_{1,rrz}\psi_{1,zzz}dx=\intop_\Omega(\psi_{1,rrz}\psi_{1,zz})_{,z}dx-\intop_\Omega\psi_{1,rrzz}\psi_{1,zz}dx,
    \label{4.11}
  \end{equation}
  where the first term vanishes due to (\ref{1.20}). Integrating the last integral in (\ref{4.11}) by parts with respect to $r$ gives
  $$
    -\intop_\Omega(\psi_{1,rzz}\psi_{1,zz}r)_{,r}drdz+\intop_\Omega\psi_{1,rzz}^2dx+\intop_\Omega\psi_{1,rzz}\psi_{1,zz}drdz,
  $$
  where the first integral vanishes, since $\psi_{1,rzz}|_{r=0}=\psi_{1,zz}|_{r=R}=0$ (recall (\ref{1.21}) and (\ref{1.18})). Thus, (\ref{4.10}) becomes
  \begin{equation}\eqal{
      &\intop_\Omega(\psi_{1,rzz}^2+\psi_{1,zzz}^2)dx+\intop_\Omega(\psi_{1,rzz}\psi_{1,zz}+3\psi_{1,rz}\psi_{1,zzz})drdz\cr
      &=-\intop_\Omega\Gamma_{,z}\psi_{1,zzz}dx.\cr}
    \label{4.12}
  \end{equation}
  Integrating by parts with respect to $z$ in the last term on the left-hand side of (\ref{4.12}) and using that $\psi_{1,zz}|_{S_2}=0$ (recall (\ref{1.20})) we get
  \begin{equation}
    \intop_\Omega(\psi_{1,rzz}^2+\psi_{1,zzz}^2)dx-\intop_\Omega\partial_r\psi_{1,zz}^2drdz=-\intop_\Omega \Gamma_{,z}\psi_{1,zzz}dx.
    \label{4.13}
  \end{equation}
  Recalling (\ref{1.18}) that $\psi_{1,zz}|_{r=R}=0$ and using the Young inequality to absorb $\psi_{1,zzz}$ we obtain
  $$
    \intop_\Omega(\psi_{1,rzz}^2+\psi_{1,zzz}^2)dx+\intop_{-a}^a\psi_{1,zz}^2|_{r=0}dz\le c|\Gamma_{,z}|_{2,\Omega}^2,
  $$
  which gives (\ref{4.7}).

  As for (\ref{4.8}), we differentiate $(\ref{1.18})_1$ with respect to $z$, multiply by $\psi_{1,rrz}$ and integrate over $\Omega$ to obtain
  \begin{equation}
    -\intop_\Omega\bigg(\psi_{1,rrz}^2+\psi_{1,zzz}\psi_{1,rrz}+3{1\over r}\psi_{1,rz}\psi_{1,rrz}\bigg)dx=\intop_\Omega\Gamma_{,z}\psi_{1,rrz}dx.
    \label{4.14}
  \end{equation}
  We integrate the second term on the left-hand side by parts in $z$, and recall (\ref{1.20}) that $\psi_{1,zz}|_{S_2}=0$, to get
  $$\eqal{
      &-\intop_\Omega\psi_{1,zzz}\psi_{1,rrz}dx=\intop_\Omega\psi_{1,zz}\psi_{1,rrzz}dx\cr
      &=\intop_\Omega(\psi_{1,zz}\psi_{1,rzz}r)_{,r}drdz-\intop_\Omega\psi_{1,rzz}^2dx-\intop_\Omega\psi_{1,zz} \psi_{1,rzz}drdz.\cr}
  $$
  We note that the first term on the right-hand side vanishes since $\psi_{1,rzz}|_{r=0}=0$ (recall (\ref{1.21})) and $\psi_{1,zz}|_{r=R}=0$ (recall (\ref{1.18})), and so (\ref{4.14}) becomes
  \begin{equation}\eqal{
      &\intop_\Omega(\psi_{1,rrz}^2+\psi_{1,rzz}^2)dx+\intop_\Omega(\psi_{1,zz}\psi_{1,rzz}+3\psi_{1,rz}\psi_{1,rrz})drdz\cr
      &=-\intop_\Omega\Gamma_{,z}\psi_{1,rrz}dx.\cr}
    \label{4.15}
  \end{equation}
  Since the first integral from the second term from the above l.h.s. equals
  $$
    {1\over 2}\intop_{-a}^a[\psi_{1,zz}]|_{r=0}^{r=R}dz=-{1\over 2}\intop_{-a}^a\psi_{1,zz}^2|_{r=0}dz
  $$
  (as $\psi_{1,zz}|_{r=R}=0$, recall (\ref{1.18})), and the last integral from the l.h.s. of (\ref{4.15}) equals
  $$
    {3\over 2}\intop_\Omega\partial_r\psi_{1,rz}^2drdz={3\over 2}\intop_{-a}^a[\psi_{1,rz}^2]_{r=0}^{r=R}dz={3\over 2}\intop_{-a}^a\psi_{1,rz}^2|_{r=R}dz
  $$
  (as $\psi_{1,rz}|_{r=0}=0$, recall (\ref{1.21})). Hence (\ref{4.15}) becomes
  \begin{equation}\eqal{
      &\intop_\Omega(\psi_{1,rrz}^2+\psi_{1,rzz}^2)dx+\intop_{-a}^a\bigg(-{1\over 2}\psi_{1,zz}^2|_{r=0}+{3\over 2}\psi_{1,rz}^2|_{r=R}\bigg)dz\cr
      &=-\intop_\Omega\Gamma_{,z}\psi_{1,rrz}dx.\cr}
    \label{4.16}
  \end{equation}
  We now use the Young inequality to absorb $\psi_{1,rrz}$ by the left-hand side to obtain (\ref{4.8}) using also (\ref{4.7}), which in turn implies (\ref{4.9}) by differentiating $(\ref{1.18})_1$ in $z$.
\end{proof}

\section{Estimates for swirl $u=rv_\varphi$}\label{s5}

We derive energy estimate for $\nabla u$. Recall that swirl $u=rv_\varphi$ satisfies
\begin{equation}\eqal{
  &u_{,t}+v\cdot\nabla u-\nu\Delta u+{2\nu\over r}u_{,r}=f_0,\cr
  &u=0\quad &{\rm on}\ \ S_1,\cr
  &u_{,z}=0\quad &{\rm on}\ \ S_2.\cr}
  \label{5.1}
\end{equation}

\begin{lemma}[see Lemma 5.1 in \cite{Z1,Z2,OZ}]\label{l5.1}
  Any regular solution $u$ to (\ref{5.1}) satisfies
  \begin{equation}
    |u_{,z}(t)|_{2,\Omega}^2+\nu|\nabla u_{,z}|_{2,\Omega^t}^2\le cD_4^2,
    \label{5.2}
  \end{equation}
  \begin{equation}
    |u_{,r}(t)|_{2,\Omega}^2+\nu(|u_{,rr}|_{2,\Omega^t}^2+|u_{,rz}|_{2,\Omega^t}^2)\le cD_5^2,
    \label{5.3}
  \end{equation}
  where $D_4^2={1\over\nu}(D_1^2+D_2^2+|u_{,z}(0)|_{2,\Omega}^2+|f_0|_{2,\Omega}^2)$ and $D_5$ is defined in (\ref{5.11}).
\end{lemma}

\begin{proof}
  Differentiating (\ref{5.1}) with respect to $z$, multiplying by $u_{,z}$ and integrating over $\Omega$ we obtain
  \begin{equation}\eqal{
    &{1\over 2}{d\over dt}|u_{,z}|_{2,\Omega}^2-\nu\intop_\Omega\divv(\nabla u_{,z}u_{,z})dx+\nu\intop_\Omega|\nabla u_{,z}|^2dx\cr
    &\quad+2\nu\intop_\Omega u_{,zr}u_{,z}drdz+\intop_\Omega v_{,z}\nabla uu_{,z}dx+{1\over 2}\intop_\Omega u\cdot\nabla(u_{,z}^2)dx\cr
    &=\intop_\Omega f_{0,z}u_{,z}dx.\cr}
    \label{5.4}
  \end{equation}
  The second term vanishes due to the boundary condition $u_{,z}|_S=0$ (recall $(\ref{1.2})_{1,2}$. The fourth term equals
  $$
    \nu\intop_\Omega\partial_r(u_{,z}^2)drdz=\nu\intop_{-a}^au_{,z}^2|_{r=0}^{r=R}dz=0,
  $$
  because $u_{,z}|_{r=R}=0$ (see $(\ref{1.2})_{1,2}$) and $u_{,z}|_{r=0}=0$ (recall (\ref{1.21})). Similarly, the sixth term equals
  $$
    {1\over 2}\intop_\Omega v\cdot\nabla u_{,z}^2dx={1\over 2}\intop_Sv\cdot\bar nu_{,z}^2dS=0,
  $$
  because $v\cdot\bar n|_S=0$ (recall $(\ref{1.2})_{1,2}$). Integrating by parts in the fifth term in (\ref{5.4}), and noting that the boundary term vanishes (since $u_{,z}=0$ on $S$), we obtain
  $$
    \bigg|\intop_\Omega(v_{,z}\cdot\nabla)u\cdot u_{,z}dx\bigg|=\bigg|\intop_\Omega v_{,z}\cdot\nabla u_{,z}udx\bigg|\le\delta\intop_\Omega|\nabla u_{,z}|^2dx+{c\over\delta}|u|_{\infty,\Omega}^2\intop_\Omega v_{,z}^2dx.
  $$
  Finally, integrating the right-hand side of (\ref{5.4}) by parts in $z$ we obtain
  $$
    \intop_\Omega f_{0,z}u_{,z}dx=-\intop_\Omega f_0u_{,zz}dx\le\delta|u_{,zz}|_{2,\Omega}^2+{c\over\delta}|f_0|_{2,\Omega}^2,
  $$
  where we used that $\intop_{S_2}|f_0u_{,z}|_{z=-a}^{z=a}rdr=0$ because $u_{,z}|_{S_2}=0$ (recall $(\ref{1.2})_2$). Using the above results in (\ref{5.4}) gives
  $$
    {d\over dt}|u_{,z}|_{2,\Omega}^2+\nu|\nabla u_{,z}|_{2\Omega}^2\le{c\over\nu}(|u|_{\infty,\Omega}^2|v_{,z}|_{2,\Omega}^2+|f_0|_{2,\Omega}^2).
  $$
  Integrating in $t\in(0,T)$ gives
  \begin{equation}\eqal{
      &|u_{,z}(t)|_{2,\Omega}^2+\nu|\nabla u_{,z}|_{2,\Omega^t}^2\cr
      &\le{c\over\nu}|u|_{\infty,\Omega^t}^2|v_{,z}|_{2,\Omega^t}^2+|u_{,z}(0)|_{2,\Omega}^2+|f_0|_{2,\Omega^t}^2)\cr
      &\le cD_4^2\cr}
    \label{5.5}
  \end{equation}
  which proves (\ref{5.2}) because energy estimate (\ref{2.5}) and maximum principle (\ref{2.10}) for the swirl $u$ were used. To prove (\ref{5.3}) we differentiate $(\ref{5.1})_1$ with respect to $r$, multiply the resulting equation by $u_{,r}$ and integrate over $\Omega$ to obtain
  \begin{equation}\eqal{
      &{1\over 2}{d\over dt}|u_{,r}|_{2,\Omega}^2+\intop_\Omega v_{,r}\cdot\nabla uu_{,r}dx+\intop_\Omega v\cdot\nabla u_{,r}u_{,r}dx\cr
      &\quad-\nu\intop_\Omega(\Delta u)_ru_{,r}dx+2\nu\intop_\Omega u_{,rr}u_{,r}drdz-2\nu\intop_\Omega{u_{,r}^2\over r^2}dx\cr
      &=\intop_\Omega f_{0,r}u_{,r}dx.\cr}
    \label{5.6}
  \end{equation}
  We now examine particular terms in (\ref{5.6}). The second term equals
  $$\eqal{
      &\intop_\Omega v_{,r}\cdot\nabla uu_{,r}rdrdz=\intop_\Omega(rv_{r,r}u_{,r}+rv_{z,r}u_{,z})u_{,r}drdz\cr
      &=-\intop_\Omega[(rv_{r,r}u_{,r})_{,r}+(rv_{z,r}u_{,r})_{,z}]udrdz\equiv I,\cr}
  $$
  where we integrated by parts with respect to $r$ and $z$, respectively, and used the boundary conditions $u|_{S_1}=u|_{r=0}=0$ (recall (\ref{1.2}) and (\ref{1.21})) and $v_{z,r}|_{S_2}=0$ (recall (\ref{1.2})). Continuing, we have
  $$\eqal{
      I&=-\intop_\Omega[(rv_{r,r})_{,r}+(rv_{z,r})_{,z}]u_{,r}udrdz\cr
      &\quad-\intop_\Omega[rv_{r,r}u_{,rr}+rv_{z,r}u_{,rz}]udrdz=I_1+I_2.\cr}
  $$
  Differentiating the divergence-free equation $(\ref{1.7})_4$ in $r$ gives $v_{r,rr}+v_{z,zr}+{v_{r,r}\over r}-{v_r\over r^2}=0$. Hence $I_1$ equals
  $$
    I_1=-\intop_\Omega{v_r\over r}u_{,r}udrdz.
  $$
  Using the Young inequality in $I_2$ yields
  $$
    |I_2|\le{\nu\over 2}(|u_{,rr}|_{2,\Omega}^2+ |u_{,rz}|_{2\Omega}^2)+{c\over\nu}|u|_{\infty,\Omega}^2(|v_{r,r}|_{2,\Omega}^2+|v_{z,r}|_{2,\Omega}^2).
  $$
  The third term on the l.h.s. of (\ref{5.6}) equals
  $$
    {1\over 2}\intop_\Omega v\cdot\nabla u_{,r}^2dx={1\over 2}\intop_\Omega\divv(vu_{,r}^2)dx=0,
  $$
  since $v\cdot\bar n|_S=0$. As for the fourth term in (\ref{5.6}) we have
  $$\eqal{
      &-\intop_\Omega(\Delta u)_{,r}u_{,r}dx=-\intop_\Omega\bigg(u_{,rrr}+\bigg({1\over r}u_{,r}\bigg)_{,r} +u_{,rzz}\bigg)u_{,r}rdrdz\cr
      &=-\intop_\Omega\bigg[\bigg(u_{,rr}+{1\over r}u_{,r}\bigg)u_{,r}r\bigg]_{,r}drdz+\intop_\Omega u_{,rr}(u_{,r}r)_{,r}drdz\cr
      &\quad+\intop_\Omega{1\over r}u_{,r}(u_{,r}r)_{,r}drdz+\intop_\Omega u_{,rz}^2dx\cr
      &=-\intop_{-a}^a\bigg[\bigg(u_{,rr}+{1\over r}u_{,r}\bigg)u_{,r}r\bigg]\bigg|_{r=0}^{r=R}dz+ \intop_\Omega(u_{,rr}^2+u_{,rz}^2)dx\cr
      &\quad+\intop_\Omega{u_{,r}^2\over r^2}dx+2\intop_\Omega u_{,rr}u_{,r}drdz.\cr}
  $$
  Using the above expressions in (\ref{5.6}) yields
  \begin{equation}\eqal{
    &{1\over 2}{d\over dt}|u_{,r}|_{2,\Omega}^2+{\nu\over 2}\intop_\Omega(u_{,rr}^2+u_{,rz}^2)dx-\nu\intop_\Omega{u_{,r}^2\over r^2}dx\cr
    &\quad-\nu\intop_{-a}^a\bigg[\bigg(u_{,rr}+{1\over r}u_{,r}\bigg)u_{,r}r\bigg]\bigg|_{r=0}^{r=R}dz+4\nu\intop_\Omega u_{,rr}u_{,r}drdz\cr
    &\le\intop_\Omega f_{0,r}u_{,r}dx+\intop_\Omega{v_r\over r}{u_{,r}\over r}udx+{c\over\nu}|u|_{\infty,\Omega}^2(|v_{r,r}|_{2,\Omega}^2+|v_{z,r}|_{2,\Omega}^2).\cr}
    \label{5.7}
  \end{equation}
  The last term on the l.h.s. of (\ref{5.7}) equals
  $$
    2\nu\intop_{-a}^au_{,r}^2\bigg|_{r=0}^{r=R}dz=2\nu\intop_{-a}^au_{,r}^2\bigg|_{r=R}dz.
  $$
  Since the expansion (\ref{1.21}) implies that
  \begin{equation}
    u=b_1(z,t)r^2+b_2(z,t)r^3+\dots,
    \label{5.8}
  \end{equation}
  thus that $u_{,r}|_{r=0}=0$.

  Projecting $(\ref{5.1})_1$ on $S_1$ implies
  \begin{equation}\eqal{
      &-\nu\intop_{-a}^a\bigg(u_{,rr}+{1\over r}u_{,r}\bigg)u_{,r}r\bigg|_{r=R}dz=-2\nu\intop_{-a}^au_{,r}^2\bigg|_{r=R}dz\cr
      &\quad+\intop_{-a}^af_0u_{,r}r\bigg|_{r=R}dz,\cr}
    \label{5.9}
  \end{equation}
  where in the last equality we used $(\ref{1.9})_1$ projected onto $S_1$. Integration by parts in $r$ in the first term on the r.h.s. of (\ref{5.7}) gives
  $$
    \intop_{-a}f_0u_{,r}r|_{r=R}dz-\intop_\Omega f_0u_{,rr}dx-\intop_\Omega f_0u_{,r}drdz,
  $$
  where we used (\ref{5.8}) again to note that $u_{,r}|_{r=0}=0$. We note again that the first term above cancells with the last term of (\ref{5.9}), while the remaining terms can be estimated using the Young inequality by
  $$
    {\nu\over 4}|u_{,rr}|_{2,\Omega}^2+\nu\bigg|{u_{,r}\over r}\bigg|_{2,\Omega}^2+{c\over\nu}|f_0|_{2,\Omega}^2.
  $$
  Using the above estimates in (\ref{5.7}) and simplifying we get
  \begin{equation}\eqal{
    &{1\over 2}{d\over dt}|u_{,r}|_{2,\Omega}^2+{\nu\over 4}(|u_{,rr}|_{2,\Omega}^2+|u_{,rz}|_{2,\Omega}^2)\le 2\nu\intop_\Omega{u_{,r}^2\over r^2}dx\cr
    &\quad+\intop_\Omega{v_r\over r}{u_{,r}\over r}udx+c|u|_{\infty,\Omega}^2(|v_{r,r}|_{2,\Omega}^2+|v_{z,r}|_{2,\Omega}^2)+{c\over\nu}|f_0|_{2,\Omega}^2.\cr}
    \label{5.10}
  \end{equation}
  Integrating (\ref{5.10}) with respect to time yields
  \begin{equation}\eqal{
      &|u_{,r}(t)|_{2,\Omega}^2+\nu(|u_{,rr}|_{2,\Omega^t}^2+|u_{,rz}|_{2,\Omega^t}^2)\cr
      &\le cD_1^2(1+D_2)+cD_1^2D_2^2+|f_0|_{2,\Omega^t}^2\cr
      &\quad+|u_{,r}(0)|_{2,\Omega}^2\equiv cD_5^2.\cr}
    \label{5.11}
  \end{equation}
  This implies (\ref{5.3}) and concludes the proof.
\end{proof}

\section{Order reduction estimates}\label{s6}

\begin{lemma}\label{l6.1}
  (See \cite{Z1,Z2,OZ}) Any regular solution to (\ref{1.1})--(\ref{1.3}) satisfies
  \begin{equation}\eqal{
    &\|\omega_r\|_{V(\Omega^t)}^2+\|\omega_z\|_{V(\Omega^t)}^2+\bigg|{\omega_r\over r}\bigg|_{2,\Omega^t}^2\cr
    &\le{1\over\nu}\phi(D_1,D_2,D_4,D_5)\bigg({R^{\varepsilon_0}\over\varepsilon_0} |v_\varphi|_{\infty,\Omega^t}^{\varepsilon_0}+{R^{2\varepsilon_0}\over\varepsilon_0^2} |v_\varphi|_{\infty,\Omega^t}^{2\varepsilon_0}\bigg)|\nabla\Gamma|_{2,\Omega^t}\cr
    &\quad+cD_6^2,\cr}
    \label{6.1}
  \end{equation}
  where
  \begin{equation}\eqal{
      D_6^2&=(D_4+D_5)\|f_\varphi\|_{L_2(0,t;L_3(S_1))}\cr
      &\quad+{1\over\nu}(|F_r|_{6/5,2,\Omega^t}^2+|F_z|_{6/5,2,\Omega^t}^2)+|\omega_r(0)|_{2,\Omega}^2+ |\omega_z(0)|_{2,\Omega}^2.\cr}
    \label{6.2}
  \end{equation}
\end{lemma}

\begin{proof}
  Multiplying $(\ref{1.8})_1$ by $\omega_r$, $(\ref{1.8})_3$ by $\omega_z$, adding the resulting equations and integrating over $\Omega^t$, we obtain
  \begin{equation}\eqal{
    &{1\over 2}(|\omega_r(t)|_{2,\Omega}^2+|\omega_z(t)|_{2,\Omega}^2)+\nu\bigg(|\nabla\omega_r|_{2,\Omega^t}^2+|\nabla\omega_z|_{2,\Omega^t}^2+\bigg|{\omega_r\over r}\bigg|_{2,\Omega^t}^2\bigg)\cr
    &=\nu\intop_{S^t}(\bar n\cdot\nabla\omega_z\omega_z+\bar n\cdot\nabla\omega_r\omega_r)dSdt'\cr
    &\quad+\intop_{\Omega^t}(v_{r,r}\omega_r^2+v_{z,z}\omega_z^2+(v_{r,z}+v_{z,r})\omega_r\omega_z)dxdt'\cr
    &\quad+\intop_{\Omega^t}(F_r\omega_r+F_z\omega_zdxdt'+{1\over 2}(|\omega_r(0)|_{2,\Omega}^2+|\omega_z(0)|_{2,\Omega}^2)\cr&
    \equiv I_1+J+I_2+{1\over 2}(|\omega_r(0)|_{2,\Omega}^2+|\omega_z(0)|_{2,\Omega}^2).\cr}
    \label{6.3}
  \end{equation}
  First we examine $I_1$. Since $\omega_r=-v_{\varphi,z},v_\varphi|_{r=R}=0$ and $v_{\varphi,z}|_{S_2}=0$ we obtain
  $$
    \intop_S\bar n\cdot\nabla\omega_r\omega_rdS=0.
  $$
  Using $(\ref{1.5})_3$ we get $\omega_z=v_{\varphi,r}+{v_\varphi\over r}$. Since $v_{\varphi,z}|_{S_2}=0$ we have
  $$\eqal{
      &-\nu\intop_{S^t}\bar n\cdot\nabla\omega_z\omega_zdS_1dt'=-\nu\intop_{S_1^t}\bar n\cdot\nabla\omega_z\omega_zdS_1dt'\cr
      &=-\nu\intop_{S_1^t}\partial_r\bigg(v_{\varphi,r}+{v_\varphi\over r}\bigg)\bigg(v_{\varphi,r}+{v_\varphi\over r}\bigg)Rdzdt'\cr
      &=-\nu\intop_0^t\intop_{-a}^a\bigg(v_{\varphi,rr}+{v_{\varphi,r}\over r}\bigg)v_{\varphi,r}\bigg|_{r=R}Rdzdt'\equiv I_1^1,\cr}
  $$
  where we used that $v_\varphi|_{S_1}=0$. Projecting $(\ref{1.7})_2$ on $S_1$ yields
  $$
    -\nu\bigg(v_{\varphi,rr}+{1\over r}v_{\varphi,r}\bigg)=f_\varphi.
  $$
  Hence
  $$
    I_1^1=R\intop_0^t\intop_{-a}^af_\varphi v_{\varphi,r}|_{r=R}dzdt'=\intop_0^t\intop_{-a}^af_\varphi\bigg(u_{,r}-{1\over R}u\bigg)dzdt'
  $$
  and
  $$
    |I_1^1|\le|f_\varphi|_{2,S_1^t}(|u_{,r}|_{2,S_1^t}+|u|_{2,S_1^t})\le c|f_\varphi|_{2,S_1^t}(D_4+D_5).
  $$
  Summarizing,
  \begin{equation}
    I_1\le c|f_\varphi|_{2,S_1^t}(D_4+D_5).
    \label{6.4}
  \end{equation}
  Next, we examine $I_2$. By the H\"older inequality, we get
  \begin{equation}\eqal{
    I_2&\le\varepsilon(|\omega_r|_{6,2,\Omega^t}^2+|\omega_z|_{6,2,\Omega^t}^2)+{1\over 4\varepsilon}(|F_r|_{6/5,2,\Omega^t}^2+|F_z|_{6/5,2,\Omega^t}^2).\cr}
    \label{6.5}
  \end{equation}
  Finally, we examine
  \begin{equation}
    J=\intop_{\Omega^t}[v_{r,r}\omega_r^2+v_{z,z}\omega_z^2+(v_{r,z}+v_{z,r})\omega_r\omega_z]dxdt'.
    \label{6.6}
  \end{equation}
  Using (\ref{1.5}) and (\ref{1.16}) yields
  \begin{equation}\eqal{
      J&=\intop_{\Omega^t}\bigg[-\psi_{,zr}\bigg({1\over r}u_{,z}\bigg)^2+\bigg(\psi_{,rz}+{\psi_{,z}\over r}\bigg)\bigg({1\over r}u_{,r}\bigg)^2\cr
        &\quad-\bigg(-\psi_{,zz}+\psi_{,rr}+{1\over r}\psi_{,r}-{\psi\over r^2}\bigg)\bigg({1\over r}u_{,z}\bigg)\bigg({1\over r}u_{,r}\bigg)\bigg]dxdt'\equiv J_1+J_2+J_3.\cr}
    \label{6.7}
  \end{equation}
  Consider $J_1$. Integrating by parts with respect to $z$ and using that $u_{,z}|_{S_2}=0$, we obtain
  $$\eqal{
      J_1&=-\intop_{\Omega^t}\psi_{,zr}{1\over r}u_{,z}{1\over r}u_{,z}dxdt'=\intop_{\Omega^t}\psi_{,zzr}{1\over r^2}u_{,z}udxdt'\cr
      &\quad+\intop_{\Omega^t}\psi_{,zr}{1\over r^2}u_{,zz}udxdt'\equiv J_{11}+J_{12}.\cr}
  $$
  Using the transformation $\psi=r\psi_1$, we have
  $$
    J_{11}=\intop_{\Omega^t}\bigg({\psi_{1,zz}\over r}+\psi_{1,zzr}\bigg){u_{,z}\over r}udxdt'\equiv J_{11}^1+J_{11}^2.
  $$
  By the H\"older inequality,
  $$\eqal{
    |J_{11}^1|&\le\intop_{\Omega^t}\bigg|{\psi_{1,zz}\over r^{1-\varepsilon_0}}\bigg|\,\bigg|{u_{,z}\over r}\bigg|\,|v_\varphi|^{\varepsilon_0}|u|^{1-\varepsilon_0}dxdt'\cr
    &\le D_2^{1-\varepsilon_0}|v_\varphi|_{\infty,\Omega^t}^{\varepsilon_0}\bigg|{u_{,z}\over r}\bigg|_{2,\Omega^t}\bigg|{\psi_{1,zz}\over r^{1-\varepsilon_0}}\bigg|_{2,\Omega^t},\cr}
  $$
  where (\ref{2.10}) is used. In view of (\ref{2.5})
  $$
    \bigg|{u_{,z}\over r}\bigg|_{2,\Omega^t}\le|v_{\varphi,z}|_{2,\Omega^t}\le D_1
  $$
  and (\ref{2.1}), (\ref{4.7}) imply
  $$
    \bigg|{\psi_{1,zz}\over r^{1-\varepsilon_0}}\bigg|\le c{R^{\varepsilon_0}\over\varepsilon_0}|\psi_{1,zzr}|_{2,\Omega^t}\le c{R^{\varepsilon_0}\over\varepsilon_0}|\Gamma_{,z}|_{2,\Omega^t}.
  $$
  Summarizing,
  $$
    |J_{11}^1|\le c{R^{\varepsilon_0}\over\varepsilon_0}D_1D_2^{1-\varepsilon_0}|v_\varphi|_{\infty,\Omega^t}^{\varepsilon_0} |\Gamma_{,z}|_{2,\Omega^t}.
  $$
  Next,
  $$
    |J_{11}^2|\le|u|_{\infty,\Omega^2}\bigg|{u_{,z}\over r}\bigg|_{2,\Omega^t}|\psi_{1,zzr}|_{2,\Omega^t}\le cD_2D_1|\Gamma_{,z}|_{2,\Omega^t},
  $$
  where (\ref{2.5}), (\ref{2.10}), (\ref{4.7}) were used. Hence,
  \begin{equation}
    |J_{11}|\le cD_1D_2^{1-\varepsilon_0}{R^{\varepsilon_0}\over\varepsilon_0} |v_\varphi|_{\infty,\Omega^t}^{\varepsilon_0}|\Gamma_{,z}|_{2,\Omega^t}+cD_1D_2|\Gamma_{,z}|_{2,\Omega^t}.
    \label{6.8}
  \end{equation}
  Next,
  $$
    J_{12}=\intop_{\Omega^t}\bigg({\psi_{1,z}\over r^2}+{\psi_{1,zr}\over r}\bigg)u_{,zz}udxdt\equiv J_{12}^1+J_{12}^2.
  $$
  Estimates (\ref{2.10}), (\ref{4.9}), (\ref{5.1}) imply
  $$
    |J_{12}^2|\le|u|_{\infty,\Omega^t}|u_{,zz}|_{2,\Omega^t}\bigg|{\psi_{1,zr}\over r}\bigg|_{2,\Omega^t}\le cD_2D_4|\Gamma_{,z}|_{2,\Omega^t}.
  $$
  Hence,
  \begin{equation}
    |J_{12}|\le\bigg|\intop_{\Omega^t}{\psi_{1,z}\over r^2}u_{,zz}udxdt'\bigg|+cD_2D_4|\Gamma_{,z}|_{2,\Omega^t}.
    \label{6.9}
  \end{equation}
  Definition of $J_1$ and (\ref{6.8}), (\ref{6.9}) imply
  \begin{equation}\eqal{
      |J_1|&\le c\bigg[{R^{\varepsilon_0}\over\varepsilon_o}D_1D_2^{1-\varepsilon_0}|v_\varphi|_{\infty,\Omega^t}^{\varepsilon_0} +D_1D_2+D_2D_4\bigg]|\Gamma_{,z}|_{2,\Omega^t}\cr
      &\quad+\bigg|\intop_{\Omega^t}{\psi_{1,z}\over r^2}u_{,zz}udxdt\bigg|.\cr}
    \label{6.10}
  \end{equation}
  Next, we estimate $J_2$. We can write it in the form
  $$
    J_2=\intop_{\Omega^t}\bigg(\psi_{,rz}+{\psi_{,z}\over r}\bigg){1\over r}u_{,r}u_{,r}drdzdt'.
  $$
  Integrating by parts with respect to $r$ yields
  $$\eqal{
      J_2&=\intop_0^t\intop_{-a}^a\bigg(\psi_{,rz}+{\psi_{,z}\over r}\bigg){1\over r}u_{,r}u\bigg|_{r=0}^{r=R}dzdt'-\intop_{\Omega^t}\bigg(\psi_{,rz}+{\psi_{,z}\over r}\bigg)_{,r}{1\over r}u_{,r}udrdzdt'\cr
      &\quad-\intop_{\Omega^t}\bigg(\psi_{,rz}+{\psi_{,z}\over r}\bigg)\,\bigg({1\over r}u_{,r}\bigg)_{,r}{u\over r}dxdt'\equiv J_{20}+J_{21}+J_{22},\cr}
  $$
  where the boundary term vanishes because $u|_{r=R}=0$ and (\ref{1.21}) implies
  $$
    \bigg(\psi_{,rz}+{\psi_{,z}\over r}\bigg){1\over r}u_{,r}u\bigg|_{r=0}=2d_{1,z}\bigg(v_{\varphi,r}+{v_\varphi\over r}\bigg)rv_\varphi\bigg|_{r=0}=4d_{1,z}b_1r^2b_1|_{r=0}=0.
  $$
  Using the transformation $\psi=r\psi_1$ in $J_{21}$ yields
  $$\eqal{
      J_{21}&=-\intop_{\Omega^t}(2\psi_{1,z}+r\psi_{1,rz})_{,r}{1\over r}{1\over r}u_{,r}udxdt'\cr
      &=-\intop_{\Omega^t}(3\psi_{1,rz}+r\psi_{1,rrz}){1\over r}{1\over r}u_{,r}udxdt'.\cr}
  $$
  By the H\"older inequality, we have
  $$\eqal{
    |J_{21}|&\le 3\bigg|{\psi_{1,rz}\over r}\bigg|_{2,\Omega^t}\bigg|{1\over r}u_{,r}\bigg|_{2,\Omega^t}|u|_{\infty,\Omega^t}+|\psi_{1,rrz}|_{2,\Omega^t}\bigg|{1\over r}u_{1,r}\bigg|_{2,\Omega^t}|u|_{\infty,\Omega^t}\cr
    &\le cD_1D_2|\Gamma_{,z}|_{2,\Omega^t},\cr}
  $$
  where (\ref{2.5}), (\ref{2.10}), (\ref{4.8}) and (\ref{4.9}) were used. Next, we consider $J_{22}$. Passing to function $\psi_1$, we get
  $$
    J_{22}=-\intop_{\Omega^t}(2\psi_{1,z}+r\psi_{1,rz}){1\over r}\bigg({1\over r}u_{,r}\bigg)_{,r}udxdt'.
  $$
  Hence
  $$\eqal{
      |J_{22}|&\le 2\bigg|\intop_{\Omega^t}{\psi_{1,z}\over r}\bigg({1\over r}u_{,r}\bigg)_{,r}udxdt'\bigg|+\bigg|\intop_{\Omega^t}\psi_{1,rz}\bigg({1\over r}u_{,r}\bigg)_{,r}udxdt'\bigg|\equiv K_1+K_2,\cr}
  $$
  where $K_2$ is bounded by
  $$
    K_2\le\bigg|\intop_{\Omega^t}{\psi_{1,rz}\over r}u_{,rr}udxdt'\bigg|+\bigg|\intop_{\Omega^t}{\psi_{1,rz}\over r}{u_{,r}\over r}udxdt'\bigg|\equiv K_2^1+K_2^2.
  $$
  Using (\ref{2.5}), (\ref{2.10}), (\ref{5.2}) and (\ref{4.9}), we obtain
  $$
    |K_2^1|\le\bigg|{\psi_{1,rz}\over r}\bigg|_{2,\Omega^t}|u_{,rr}|_{2,\Omega^t}|u|_{\infty,\Omega^t}\le cD_2D_5|\Gamma_{,z}|_{2,\Omega^t}
  $$
  and
  $$
    |K_2^2|\le\bigg|{\psi_{1,rz}\over r}\bigg|_{2,\Omega^t}\bigg|{1\over r}u_{,r}\bigg|_{2,\Omega^t}|u|_{\infty,\Omega^t}\le cD_1D_2|\Gamma_{,z}|_{2,\Omega^t}.
  $$
  Summarizing,
  \begin{equation}\eqal{
      |J_2|&\le c[D_1D_2+D_2D_5]|\Gamma_{,z}|_{2,\Omega^t}+2\bigg|\intop_{\Omega^t}{\psi_{1,z}\over r}\bigg({1\over r}u_{,r}\bigg)_{,r}udxdt'\bigg|.\cr}
    \label{6.11}
  \end{equation}
  Finally, we examine $J_3$. Using that $\psi=r\psi_1$ yields
  $$
    J_3=-\intop_{\Omega^t}(-r\psi_{1,zz}+3\psi_{1,r}+r\psi_{1,rr}){1\over r}u_{,r}{1\over r}u_{,z}dxdt'.
  $$
  Integrating by parts with respect to $z$, using that $\psi_1|_{S_2}=0$ and \break$\psi_{1,zz}|_{S_2}=-\Gamma|_{S_2}=0$, we obtain
  $$\eqal{
      J_3&=\intop_{\Omega^t}\bigg(-\psi_{1,zzz}+{3\over r}\psi_{1,rz}+\psi_{1,rrz}\bigg){1\over r}u_{,r}udxdt'\cr
      &\quad+\intop_{\Omega^t}\bigg(-\psi_{1,zz}+{3\over r}\psi_{1,r}+\psi_{1,rr}\bigg)\,\bigg({1\over r}u_{,r}\bigg)_{,z}udxdt'\cr
      &\equiv J_{31}+J_{32}.\cr}
  $$
  Using (\ref{2.5}), (\ref{2.10}), (\ref{4.8}) and (\ref{4.9}), we have
  $$
    |J_{31}|\le cD_1D_2|\Gamma_{,z}|_{2,\Omega^t}.
  $$
  To estimate $J_{32}$ we recall that
  $$
    -\psi_{1,rr}-{3\over r}\psi_{1,r}-\psi_{1,zz}=\Gamma.
  $$
  Then $J_{32}$ takes the form
  $$
    J_{32}=-\intop_{\Omega^t}(2\psi_{1,zz}+\Gamma)\bigg({1\over r}u_{,r}\bigg)_{,z}udxdt'\equiv J_{32}^1+J_{32}^2.
  $$
  Continuing,
  $$\eqal{
    |J_{32}^1|&\le c\bigg|\intop_{\Omega^t}{\psi_{1,zz}\over r^{1-\varepsilon_0}}u_{,rz}u^{1-\varepsilon_0}v_\varphi^{\varepsilon_0}dxdt\bigg|\le cD_2^{1-\varepsilon_0}|v_\varphi|_{\infty,\Omega}^{\varepsilon_0}\bigg|{\psi_{1,zz}\over r^{1-\varepsilon_0}}\bigg|_{2,\Omega^t}\cdot|u_{,rz}|_{2,\Omega^t},\cr}
  $$
  where we used (\ref{2.10}). Finally, we have
  $$
    |J_{32}^1|\le cD_2^{1-\varepsilon_0}D_4{R^{\varepsilon_0}\over\varepsilon_0}|v_\varphi|_{\infty,\Omega^t}^{\varepsilon_0} |\Gamma_{,z}|_{2,\Omega^t}.
  $$
  Next,
  $$\eqal{
    |J_{32}^2|&\le\intop_{\Omega^t}\bigg|{\Gamma\over r^{1-\varepsilon_0}}\bigg||u_{,rz}|\,|u|^{1-\varepsilon_0}|v_\varphi|^{\varepsilon_0}dxdt'\le cD_2^{1-\varepsilon_0}D_4{R^{\varepsilon_0}\over\varepsilon_0}|v_\varphi|_{\infty,\Omega^t}^{\varepsilon_0} |\Gamma_{,r}|_{2,\Omega^t}.\cr}
  $$
  Summarizing,
  \begin{equation}
    |J_3|\le cD_1D_2|\Gamma_{,z}|_{2,\Omega^t}+cD_2^{1-\varepsilon_0}D_4{R^{\varepsilon_0}\over\varepsilon_0}|v_\varphi|_{\infty,\Omega^t}^{\varepsilon_0} |\nabla\Gamma|_{2,\Omega^t}.
    \label{6.12}
  \end{equation}
  Using estimates (\ref{6.10}), (\ref{6.11}), (\ref{6.12}) in (\ref{6.7}) implies
  \begin{equation}\eqal{
      |J|&\le\bigg|\intop_{\Omega^t}{\psi_{1,z}\over r^2}u_{,zz}udxdt\bigg|+2\bigg|\intop_{\Omega^t}{\psi_{1,z}\over r}\bigg({1\over r}u_{,r}\bigg)_{,r}udxdt\bigg|\cr
      &\quad+c\bigg[{R^{\varepsilon_0}\over\varepsilon_0}(D_1+D_4)D_2^{1-\varepsilon_0} |v_\varphi|_{\infty,\Omega^t}^{\varepsilon_0}\cr
        &\quad+D_1D_2+D_2D_4+D_2D_5\bigg]\,|\nabla\Gamma|_{2,\Omega^t}.\cr}
    \label{6.13}
  \end{equation}
  Using estimates (\ref{6.4}), (\ref{6.5}), (\ref{6.13}) in (\ref{6.3}) implies the inequality
  \begin{equation}\eqal{
      &|\omega_r(t)|_{2,\Omega}^2+|\omega_z(t)|_{2,\Omega}^2+\nu(|\nabla\omega_r|_{2,\Omega^t}^2+ |\nabla\omega_z|_{2,\Omega^t}^2+|\Phi|_{2,\Omega^t}^2)\cr
      &\le\bigg|\intop_{\Omega^t}{\psi_{1,z}\over r^2}u_{,zz}udxdt'\bigg|+2\bigg|\intop_{\Omega^t}{\psi_{1,z}\over r}\bigg({1\over r}u_{,r}\bigg)_{,r}udxdt'\bigg|\cr
      &\quad+c\bigg[{R^{\varepsilon_0}\over\varepsilon_0}(D_1+D_4)D_2^{1-\varepsilon_0} |v_\varphi|_{\infty,\Omega^t}^{\varepsilon_0}+D_1D_2+D_2(D_4+D_5)\bigg]\cdot\cr
      &\quad\cdot|\nabla\Gamma|_{2,\Omega^t}+c|f_\varphi|_{2,S_1^t}(D_4+D_5)\cr
      &\quad+c(|F_r|_{6/5,2,\Omega^t}^2+|F_z|_{6/5,2,\Omega^t}^2)+|\omega_r(0)|_{2,\Omega}^2+ |\omega_z(0)|_{2,\Omega}^2.\cr}
    \label{6.14}
  \end{equation}
  Recalling that $\omega_r=-{1\over r}u_{,z}$, $\omega_z={1\over r}u_{,r}$ (see (\ref{1.5})) and applying the H\"older and Young inequalities to the first term on the r.h.s. of (\ref{6.14}) we bound it by
  $$
    \varepsilon_1|\omega_{r,z}|_{2,\Omega^t}^2+{1\over 4\varepsilon_1}\intop_{\Omega^t}{\psi_{1,z}^2\over r^2}u^2dxdt'\equiv L_1.
  $$
  The second term in $L_1$ can be written in the form
  $$
    \intop_{\Omega^t}{\psi_{1,z}^2\over r^{2(1-\varepsilon_0)}}{u^2\over r^{2\varepsilon_0}}dxdt'\le D_2^{2(1-\varepsilon_0)}|v_\varphi|_{\infty,\Omega^t}^{2\varepsilon_0}\intop_{\Omega^t} {\psi_{1,z}^2\over r^{2(1-\varepsilon_0)}}dxdt',
  $$
  where $\varepsilon_0$ can be chosen as small as we need. By the Hardy inequality
  $$
    \intop_{\Omega^t}{\psi_{1,z}^2\over r^{2(1-\varepsilon_0)}}dxdt'\le{R^{2\varepsilon_0}\over\varepsilon_0^2}\intop_{\Omega^t}\psi_{1,rz}^2dxdt'.
  $$
  Applying the interpolation inequality (\ref{2.3}) (see [Ch. 3, sect. 15]\cite{BIN})
  $$
    \intop_\Omega\psi_{1,rz}^2dx\le\bigg(\intop_\Omega|\nabla^2\psi_{1,z}|^2dx\bigg)^\theta \bigg(\intop_\Omega\psi_{1,z}^2dx\bigg)^{1-\theta},
  $$
  where $\theta$ satisfies the equality
  $$
    {3\over 2}-1=(1-\theta){3\over 2}+\theta\bigg({3\over 2}-2\bigg)\quad {\rm so}\ \ \theta=1/2.
  $$
  Using (\ref{2.12}), we get
  $$
    \intop_{\Omega^t}\psi_{1,zr}^2dx\le c|\nabla^2\psi_{1,z}|_{2,\Omega}|\psi_{1,z}|_{2,\Omega}\le cD_1|\nabla^2\psi_{1,z}|_{2,\Omega}.
  $$
  Summarizing,
  \begin{equation}\eqal{
      &\bigg|\intop_{\Omega^t}{\psi_{1,z}\over r^2}u_{,zz}udxdt\bigg|\le\varepsilon_1|\omega_{r,z}|_{2,\Omega^t}^2\cr
      &\quad+{c\over 4\varepsilon_1}D_1D_2^{2(1-\varepsilon_0)}{R^{2\varepsilon_0}\over\varepsilon_0^2} |v_\varphi|_{\infty,\Omega^t}^{2\varepsilon_0}|\Gamma_{,z}|_{2,\Omega^t}.\cr}
    \label{6.15}
  \end{equation}
  Similarly, the second term on the r.h.s. of (\ref{6.14}) is bounded by
  \begin{equation}\eqal{
      &\bigg|\intop_{\Omega^t}{\psi_{1,z}\over r}\bigg({1\over r}u_{,r}\bigg)_{,r}udxdt'\bigg|\le\varepsilon_2|\omega_{z,r}|_{2,\Omega^t}^2\cr
      &\quad+{c\over 4\varepsilon_2}D_1D_2^{2(1-\varepsilon_0)}{R^{2\varepsilon_0}\over\varepsilon_0^2} |v_\varphi|_{\infty,\Omega^t}^{2\varepsilon_0}|\Gamma_{,z}|_{2,\Omega^t}.\cr}
    \label{6.16}
  \end{equation}
  Using (\ref{6.15}) and (\ref{6.16}) in (\ref{6.14}) yields (\ref{6.1}).
\end{proof}

\begin{lemma}\label{l6.2}
  Assume that $D_1$, $D_2$ are defined in Notation \ref{n1.1}, $v_\varphi(0)\in L_\infty(\Omega)$, $f_\varphi/r\in L_1(0,t;L_\infty(\Omega))$. Then
  \begin{equation}
    |v_\varphi(t)|_{\infty,\Omega}\le{D_2\over\sqrt{\nu}}D_1^{1/4}X^{3/4}+D_7,
    \label{6.17}
  \end{equation}
  where
  \begin{equation}
    D_7=\sqrt{2}D_2^{1/2}\bigg|{f_\varphi\over r}\bigg|_{\infty,1,\Omega^t}^{1/2}+|v_\varphi(0)|_{\infty,\Omega}.
    \label{6.18}
  \end{equation}
\end{lemma}

\begin{proof}
  Multiplying $(\ref{1.7})_2$ by $v_\varphi|v_\varphi|^{s-2}$ and integrating over $\Omega$ yields
  \begin{equation}\eqal{
    &{1\over s}{d\over dt}|v_\varphi|_{s,\Omega}^s+{4\nu(s-1)\over s^2}|\nabla|v_\varphi|^{s/2}|_{2,\Omega}^2+\nu\intop_\Omega{|v_\varphi|^s\over r^2}dx\cr
    &=\intop_\Omega\psi_{1,z}|v_\varphi|^sdx+\intop_\Omega f_\varphi v_\varphi^{s-1}dx.\cr}
    \label{6.19}
  \end{equation}
  The first term on the r.h.s. of (\ref{6.19}) is estimated by
  $$
    \varepsilon\intop_\Omega{|v_\varphi|^s\over r^2}dx+{D_2^2\over 4\varepsilon}\intop_\Omega|\psi_{1,z}|^2|v_\varphi|^{s-2}dx.
  $$
  The second integral on the r.h.s. of \eqref{6.19} is estimated by
  $$\eqal{
    &\intop_\Omega|f_\varphi|\,|v_\varphi|^{s-1}dx=\intop_\Omega\bigg|{f_\varphi\over r}\bigg|r|v_\varphi|^{s-1}dx\cr
    &\le D_2\intop_\Omega\bigg|{f_\varphi\over r}\bigg|\,|v_\varphi|^{s-2}dx\le D_2\bigg|{f_\varphi\over r}\bigg|_{s/2,\Omega}|v_\varphi|_{s,\Omega}^{s-2}.\cr}
  $$
  In view of the above estimates inequality (\ref{6.19}) reads
  $$
    {1\over s}{d\over dt}|v_\varphi|_{s,\Omega}^s\le{D_2^2\over 2\nu}|\psi_{1,z}|_{s,\Omega}^2|v_\varphi|_{s,\Omega}^{s-2}+D_2\bigg|{f_\varphi\over r}\bigg|_{s/2,\Omega}|v_\varphi|_{s,\Omega}^{s-2}.
  $$
  Simplifying, we get
  $$
    {d\over dt}|v_\varphi|_{s,\Omega}^2\le{D_2^2\over\nu}|\psi_{1,z}|_{s,\Omega}^2+2D_2\bigg|{f_\varphi\over r}\bigg|_{s/2,\Omega}.
  $$
  Integrating the inequality with respect to time and passing with $s$ to $\infty$ we get
  $$
    |v_\varphi|_{\infty,\Omega}^2\le{D_2^2\over\nu}\intop_0^t|\psi_{1,z}|_{\infty,\Omega}^2dt'+2D_2\bigg|{f_\varphi\over r}\bigg|_{\infty,1,\Omega^t}+|v_\varphi(0)|_{\infty,\Omega}^2.
  $$
  Using the interpolation
  $$
    |\psi_{1,z}|_{\infty,\Omega}\le|\psi_{1,z}|_{2,\Omega}^{1/4}|D^2\psi_{1,z}|_{2,\Omega}^{3/4}
  $$
  and (\ref{2.12}), (\ref{4.8}), we obtain
  $$
    |v_\varphi|_{\infty,\Omega}^2\le{D_2^2\over\nu}D_1^{1/2}|\Gamma_{,z}|_{2,\Omega^t}^{3/2}+\bigg|{f_\varphi\over r}\bigg|_{\infty,1,\Omega^t}+2D_2|v_\varphi(0)|_{\infty,\Omega}^2.
  $$
  The above inequality implies (\ref{6.17}) and concludes the proof.
\end{proof}

\begin{lemma}\label{l6.3}
  Assume that
  $$
    {|v_\varphi|_{s,\infty,\Omega^t}\over|v_\varphi|_{\infty,\Omega^t}}\ge c_0.
  $$
  Then
  \begin{equation}
    |v_\varphi(t)|_{s,\Omega}\le{1\over c_0^{s-1}}(D_1^2+|f_\varphi|_{s,1,\Omega^t})+|v_\varphi(0)|_{s,\Omega}.
    \label{6.20}
  \end{equation}
\end{lemma}

\begin{proof}
  Multiply $(\ref{1.7})_2$ by $v_\varphi|v_\varphi|^{s-2}$, integrate over $\Omega$ and use boundary conditions to obtain
  \begin{equation}\eqal{
    &{1\over s}{d\over dt}|v_\varphi|_{s,\Omega}^s+{4\nu(s-1)\over s^2}|\nabla v_\varphi^{s/2}|_{2,\Omega}^2+\nu\intop_\Omega{|v_\varphi|^s\over r^2}dx\cr
    &\le\intop_\Omega\bigg|{v_r\over r}\bigg||v_\varphi|^sdx+\intop_\Omega|f_\varphi|\,|v_\varphi|^{s-1}dx.\cr}
    \label{6.21}
  \end{equation}
  Dropping the last two terms on the l.h.s. we get
  \begin{equation}
    |v_\varphi|_{s,\Omega}^{s-1}{d\over dt}|v_\varphi|_{s,\Omega}\le|v_\varphi|_{\infty,\Omega}^{s-1}\bigg(\bigg|{v_r\over r}\bigg|_{2,\Omega}|v_\varphi|_{2,\Omega}+|f_\varphi|_{s,\Omega}\bigg).
    \label{6.22}
  \end{equation}
  Continuing,
  $$
    {d\over dt}|v_\varphi|_{s,\Omega}\le\bigg({|v_\varphi|_{\infty,\Omega}\over|v_\varphi|_{s,\Omega}}\bigg)^{s-1} \bigg(\bigg|{v_r\over r}\bigg|_{2,\Omega}|v_\varphi|_{2,\Omega}+|f_\varphi|_{s,\Omega}\bigg).
  $$
  Integrating with respect to time yields
  $$
    |v_\varphi(t)|_{s,\Omega}\le\bigg({|v_\varphi|_{\infty,\Omega^t}\over|v_\varphi|_{s,\infty,\Omega^t}}\bigg)^{s-1} \bigg(\bigg|{v_r\over r}\bigg|_{2,\Omega^t}|v_\varphi|_{2,\Omega^t}+|f_\varphi|_{s,1,\Omega^t}\bigg)+ |v_\varphi(0)|_{s,\Omega}.
  $$
  The above inequality implies (\ref{6.20}).
\end{proof}

\section{Conditional higher regularity estimate for regular solutions.}\label{s7}

Assuming the conditional bound for $X_s(t)$ obtained in Theorem \ref{t1.2}, we show that appropriate regularity of data implies
\begin{equation}
  \|v\|_{W_2^{4,2}(\Omega^t)}+\|\nabla p\|_{W_2^{2,1}(\Omega^t)}\le C,
  \label{7.1}
\end{equation}
where $C=\Psi_{s,A,c_0}\left(\mathrm{data}, \int_0^t E_{W,s}(\tau)\,d\tau\right)$.

\subsection{Preliminaries}\label{s7.1}

We first introduce some functional analytic tools to handle the anisotropic Sobolev spaces.

\begin{definition}[Anisotropic Sobolev and Sobolev-Slobodetskii spaces]\label{d7.1}
  We denote by
  \begin{itemize}
    \item[1.] $W_{p,p_0}^{k,k/2}(\Omega^t)$, $k,k/2\in\N\cup\{0\}$, $p,p_0\in[1,\infty]$ -- the anisotropic Sobolev space with a mixed norm, which is a completion of $C^\infty(\Omega^T)$-functions under the norm
          $$
            \|u\|_{W_{p,p_0}^{k,k/2}(\Omega^T)}=\bigg(\intop_0^T\bigg(\sum_{|\alpha|+2a\le k}\intop_\Omega|D_x^\alpha\partial_t^au|^p\bigg)^{p_0/p}dt\bigg)^{1/p_0}.
          $$
    \item[2.] $W_{p,p_0}^{s,s/2}(\Omega^T)$, $s\in\R_+$, $p,p_0\in[1,\infty)$ -- the Sobolev-Slobodetskii space with the finite norm
          $$\eqal{
            &\|u\|_{W_{p,p_0}^{s,s/2}(\Omega^T)}=\sum_{|\alpha|+2a\le|s|}\|D_x^\alpha\partial_t^au\|_{L_{p,p_0}(\Omega^T)}\cr
            &+\bigg[\intop_0^T\!\!\bigg(\intop_\Omega\intop_\Omega\!\sum_{|\alpha|+2a\le[s]}\!\! {|D_x^\alpha\partial_t^au(x,t)-D_{x'}^\alpha\partial_t^au(x',t)|^p\over|x-x'|^{n+p(s-[s])}}dxdx'\bigg)^{p_0/p}dt \bigg]^{1/p_0}\cr
            &+\bigg[\intop_\Omega\!\!\bigg(\intop_0^T\intop_0^T\!\sum_{|\alpha|+2a=[s]}\!\! {|D_x^\alpha\partial_t^au(x,t)-D_x^\alpha\partial_{t'}^au(x,t')|^{p_0}\over|t-t'|^{1+p_0({s\over 2}-[{s\over 2}])}}dtdt'\bigg)^{p/p_0}dx\bigg]^{1/p}\!,\cr}
          $$
          where $a\in\N\cup\{0\}$, $[s]$ is the integer part of $s$ and $D_x^\alpha$ denotes the\break partial derivative in the spatial variable $x$ corresponding to multi-\break index $\alpha$. For $s$ odd the last but one term in the above norm vanishes whereas for $s$ even the last two terms vanish. We also use notation\break $L_p(\Omega^T)=L_{p,p}(\Omega^T)$, $W_p^{s,s/2}(\Omega^T)=W_{p,p}^{s,s/2}(\Omega^T)$.
    \item[3.] $B_{p,p_0}^l(\Omega)$, $l\in\R_+$, $p,p_0\in[1,\infty)$ -- the Besov space with the finite norm
          $$
            \|u\|_{B_{p,p_0}^l(\Omega)}=\|u\|_{L_p(\Omega)}+\bigg(\sum_{i=1}^n\intop_0^\infty {\|\Delta_i^m(h,\Omega)\partial_{x_i}^ku\|_{L_p(\Omega)}^{p_0}\over h^{1+(l-k)_{p_0}}}dh\bigg)^{1/p_0},
          $$
          where $k\in\N\cup\{0\}$, $m\in\N$, $m>l-k>0$, $\Delta_i^j(h,\Omega)u$, $j\in\N$, $h\in\R_+$ is the finite difference of the order $j$ of the function $u(x)$ with respect to $x_i$ with
          $$\eqal{
              &\Delta_i^1(h,\Omega)u=\Delta_i(h,\Omega)\cr
              &=u(x_1,\dots,x_{i-1},x_i+h,x_{i+1},\dots,x_n)-u(x_1,\dots,x_n),\cr
              &\Delta_i^j(h,\Omega)=\Delta_i(h,\Omega)\Delta_i^{j-1}(h,\Omega)u\quad {\rm and}\ \ \Delta_i^j(h,\Omega)u=0\cr
              &{\rm for}\ \ x+jh\not\in\Omega.\cr}
          $$
          It has been proved in \cite{G} that the norms of the Besov space $B_{p,p_0}^l(\Omega)$ are equivalent for different $m$ and $k$ satisfying the condition $m>l-k>0$.
  \end{itemize}
\end{definition}

We need the following interpolation lemma.

\begin{lemma}[Anisotropic interpolation, see {\cite[Ch. 4, Sect. 18]{BIN}}]\label{l7.2}
  Let $u\in W_{p,p_0}^{s,s/2}(\Omega^T)$, $s\in\R_+$, $p,p_0\in[1,\infty]$, $\Omega\subset\R^3$. Let $\sigma\in\R_+\cup\{0\}$, and
  $$
    \varkappa={3\over p}+{2\over p_0}-{3\over q}-{2\over q_0}+|\alpha|+2a+\sigma<s.
  $$
  Then $D_x^\alpha\partial_t^au\in W_{q,q_0}^{\sigma,\sigma/2}(\Omega^T)$, $q\ge p$, $q_0\ge p_0$ and there exists $\varepsilon\in(0,1)$ such that
  $$
    \|D_x^\alpha\partial_t^au\|_{W_{q,q_0}^{\sigma,\sigma/2}(\Omega^T)}\le\varepsilon^{s-\varkappa} \|u\|_{W_{p,p_0}^{s,s/2}(\Omega^t)}+c\varepsilon^{-\varkappa}\|u\|_{L_{p,p_0}(\Omega^t)}.
  $$
  We recall from \cite{B} the trace and the inverse trace theorems for Sobolev spaces with a mixed norm.
\end{lemma}

\begin{lemma}\label{l7.3}
  (traces in $W_{p,p_0}^{s,s/2}(\Omega^T)$, see \cite{B})
  \begin{itemize}
    \item[(i)] Let $u\in W_{p,p_0}^{s,s/2}(\Omega^t)$, $s\in\R_+$, $p,p_0\in(1,\infty)$. Then $u(x,t_0)=u(x,t)|_{t=t_0}$ for $t_0\in[0,T]$ belongs to $B_{p,p_0}^{s-2/p_0}(\Omega)$, and
          $$
            \|u(\cdot,t_0)\|_{B_{p,p_0}^{s-2/p_0}(\Omega)}\le c\|u\|_{W_{p,p_0}^{s,s/2}(\Omega^T)},
          $$
          where $c$ does not depend on $u$.
    \item[(ii)] For given $\bar u\in B_{p,p_0}^{s-2/p_0}(\Omega)$, $s\in\R_+$, $s>2/p_0$, $p_0\in(1,\infty)$, there exists a function $u\in W_{p,p_0}^{s,s/2}(\Omega^t)$ such that $u|_{t=t_0}=\bar u$ for $t_0\in[0,T]$ and
          $$
            \|u\|_{W_{p,p_0}^{s,s/2}(\Omega^T)}\le c\|\bar u\|_{B_{p,p_0}^{s-2/p_0}(\Omega)},
          $$
          where constant $c$ does not depend on $\bar u$.
  \end{itemize}
\end{lemma}

We need the following imbeddings between Besov spaces

\begin{lemma}[see {\cite[Th. 4.6.1]{T}}]\label{l7.4}
  Let $\Omega\subset\R^n$ be an arbitrary domain.
  \begin{itemize}
    \item[(a)] Let $s\in\R_+$, $\varepsilon>0$, $p\in(1,\infty)$, and $1\le q_1\le q_2\le\infty$. Then
          $$
            B_{p,1}^{s+\varepsilon}(\Omega)\subset B_{p,\infty}^{s+\varepsilon}(\Omega)\subset B_{p,q_1}^s(\Omega)\subset B_{p,q_2}^s(\Omega)\subset B_{p,1}^{s-\varepsilon}(\Omega)\subset B_{p,\infty}^{s-\varepsilon}(\Omega).
          $$
    \item[(b)] Let $\infty>q\ge p>1$, $1\le r\le\infty$, $0\le t\le s<\infty$ and
          $$
            t+{n\over p}-{n\over q}\le s.
          $$
          then $B_{p,r}^s(\Omega)\subset B_{q,r}^t(\Omega)$.
  \end{itemize}
\end{lemma}

\begin{lemma}[see {\cite[Ch. 4, Th. 18.8]{BIN}}]\label{lemma 7.5}
  Let $1\le\theta_1<\theta_2\le\infty$. Then
  $$
    \|u\|_{B_{p,\theta_2}^l(\Omega)}\le c\|u\|_{B_{p,\theta_1}^l(\Omega)},
  $$
  where $c$ does not depend on $u$.
\end{lemma}

\begin{lemma}[see {\cite[Ch. 4, Th. 18.9]{BIN}}]\label{l7.6}
  Let $l\in\N$ and $\Omega$ satisfy the $l$-horn condition. Then the following imbeddings hold
  $$\eqal{
      &\|u\|_{B_{p,2}^l(\Omega)}\le c\|u\|_{W_p^l(\Omega)}\le c\|u\|_{B_{p,p}^l(\Omega)},\quad &1\le p\le 2,\cr
      &\|u\|_{B_{p,p}^l(\Omega)}\le c\|u\|_{W_p^l(\Omega)}\le c\|u\|_{B_{p,2}^l(\Omega)},\quad &2\le p<\infty,\cr
      &\|u\|_{B_{p,\infty}^l(\Omega)}\le c\|u\|_{W_p^l(\Omega)}\le c\|u\|_{B_{p,1}^l(\Omega)},\quad &1\le p\le\infty.\cr}
  $$
\end{lemma}

Consider the nonstationary Stokes system in $\Omega\subset\R^3$:
$$
  \begin{aligned}
     & v_t-\nu\Delta v+\nabla p=f,                             \\
     & \divv v=0,                                              \\
     & v_r=v_\varphi=\omega_\varphi=0   &  & \text{ on }S_1^T, \\
     & v_z=\omega_\varphi=v_{\varphi,z}=0 &  & \text{ on }S_2^T, \\
     & v|_{t=0}=v(0),
  \end{aligned}
$$
with the boundary conditions (\ref{1.2}) and given initial condition $v(0)$.

\begin{lemma}[see \cite{MS}]\label{l7.7}
  Assume that $f\in L_{q,r}(\Omega^T)$, $v(0)\in B_{q,r}^{2-2/r}(\Omega)$, $r,q\in(1,\infty)$. Then there exists a unique solution to the above system such that $v\in W_{q,r}^{2,1}(\Omega^T)$, $\nabla p\in L_{q,r}(\Omega^T)$ with the following estimate
  \begin{equation}\eqal{
      \|v\|_{W_{q,r}^{2,1}(\Omega^T)}+\|\nabla p\|_{L_{q,r}(\Omega^t)}&\le  c(\|f\|_{L_{q,r}(\Omega^t)}+\|v(0)\|_{B_{q,r}^{2-2/r}(\Omega)}).\cr}
    \label{7.2}
  \end{equation}
\end{lemma}

\subsection{Proof of (\ref{7.1})}\label{s7.2}

We show (\ref{7.1}) in the following series of lemmas.

\begin{lemma}\label{l7.8}
  Suppose that $X_s(t)<\infty$, i.e. that
  \begin{equation}
    \|\Phi\|_{V(\Omega^t)}+\|\Gamma\|_{V(\Omega^t)}\le\phi_1,
    \label{7.3}
  \end{equation}
  where $\phi_1=\Psi_{s,A,c_0}\left(\mathrm{data}, \int_0^t E_{W,s}(\tau)\,d\tau\right)$. Assume that
  $$
    f\in W_2^{2,1}(\Omega^t),\quad v(0)\in W_2^3(\Omega).
  $$
  Then
  \begin{equation}
    \|v\|_{W_2^{4,2}(\Omega^t)}+\|\nabla p\|_{W_2^{2,1}(\Omega^t)}\le\phi(\phi_1,\|f\|_{W_2^{2,1}(\Omega^t)},\|v(0)\|_{H^3(\Omega)}).
    \label{7.4}
  \end{equation}
\end{lemma}

\begin{proof}
  From (\ref{7.3}) we have
  \begin{equation}
    \|\Gamma\|_{V(\Omega^t)}\le\phi_1.
    \label{7.5}
  \end{equation}
  By the elliptic estimates for $\psi_1$ in Section \ref{s4}, we obtain
  \begin{equation}
    \|\psi_1\|_{2,\infty,\Omega^t}\le c\phi_1,
    \label{7.6}
  \end{equation}
  \begin{equation}
    \|\psi_1\|_{3,2,\Omega^t}\le c\phi_1.
    \label{7.7}
  \end{equation}
  From (\ref{1.19}) the following relations hold
  \begin{equation}
    v_r=-r\psi_{1,z},\quad v_z=2\psi_1+r\psi_{1,r}.
    \label{7.8}
  \end{equation}
  Hence (\ref{7.6}) and $R$ finite imply
  \begin{equation}
    \|v_r\|_{1,2,\infty,\Omega^t}+\|v_z\|_{1,2,\infty,\Omega^t}\le c\phi_1.
    \label{7.9}
  \end{equation}
  To increase regularity of $v$ we consider the Stokes problem
  \begin{equation}\eqal{
      &v_{,t}-\nu\Delta v+\nabla p=-v'\cdot\nabla v+f\quad &{\rm in}\ \ \Omega^T,\cr
      &\divv v=0\quad &{\rm in}\ \ \Omega^T,\cr
      &v\cdot\bar n=0,\ \ v_{z,r}=0,\ \ v_\varphi=0\quad &{\rm on}\ \ S_1^T,\cr
      &v\cdot\bar n=0,\ \ v_{r,z}=0,\ \ v_{\varphi,z}=0\quad &{\rm on}\ \ S_2^T,\cr
      &v|_{t=0}=v(0)\quad &{\rm in}\ \ \Omega,\cr}
    \label{7.10}
  \end{equation}
  where $v'=v_r\bar e_r+v_z\bar e_z$.

  From (\ref{7.9}) we have
  \begin{equation}
    |v'|_{6,\infty,\Omega^t}\le c\phi_1.
    \label{7.11}
  \end{equation}
  For solutions to (\ref{7.10}) the following energy estimate holds
  \begin{equation}
    \|v\|_{V(\Omega^t)}\le c(|f|_{2,\Omega^t}+|v(0)|_{2,\Omega})\equiv d_1.
    \label{7.12}
  \end{equation}
  Hence, (\ref{7.12}) yields
  \begin{equation}
    |\nabla v|_{2,\Omega^t}\le d_1.
    \label{7.13}
  \end{equation}
  Estimates (\ref{7.11}) and (\ref{7.13}) imply
  \begin{equation}
    |v'\cdot\nabla v|_{3/2,2,\Omega^t}\le c\phi_1d_1.
    \label{7.14}
  \end{equation}
  Applying \cite{MS} to (\ref{7.10}) yields
  \begin{equation}\eqal{
      &\|v\|_{W_{{3\over 2},2}^{2,1}(\Omega^t)}+|\nabla p|_{3/2,2,\Omega^t}\cr
      &\le c(|f|_{3/2,2,\Omega^t}+\|v(0)\|_{B_{3/2,2}^1(\Omega)}+\phi_1d_1)\equiv d_2.\cr}
    \label{7.15}
  \end{equation}
  In view of the imbedding (see \cite[Ch. 3, Sect. 10]{BIN})
  \begin{equation}
    |\nabla v|_{5/2,\Omega^t}\le c\|v\|_{W_{3/2,2}^{2,1}(\Omega^t)}
    \label{7.16}
  \end{equation}
  and (\ref{7.11}) we derive that
  \begin{equation}
    |v'\cdot\nabla v|_{{30\over 17},{5\over 2},\Omega^t}\le c\phi_1d_2.
    \label{7.17}
  \end{equation}
  Then applying again \cite{MS} to problem (\ref{7.10}) yields
  \begin{equation}\eqal{
      &\|v\|_{W_{{30\over 17},{5\over 2}}^{2,1}(\Omega^t)}+|\nabla p|_{{30\over 17},{5\over 2},\Omega^t}\cr
      &\le c(|f|_{{30\over 17},{5\over 2},\Omega^t}+\|v(0)\|_{B_{{30\over 17},{5\over 2}}^{2-4/5}(\Omega)}+\phi_1d_2)\equiv d_3.\cr}
    \label{7.18}
  \end{equation}
  In view of the imbedding (see \cite[Ch. 3, Sect. 10]{BIN})
  \begin{equation}
    |\nabla v|_{{10\over 3},\Omega^t}\le c\|v\|_{W_{{30\over 17},{5\over 2}}^{2,1}(\Omega^t)}
    \label{7.19}
  \end{equation}
  and (\ref{7.11}) we have
  \begin{equation}
    |v'\cdot\nabla v|_{{15\over 7},{10\over 3},\Omega^t}\le c\phi_1d_3.
    \label{7.20}
  \end{equation}
  Applying \cite{MS} to (\ref{7.10}) implies
  \begin{equation}\eqal{
      &\|v\|_{W_{{15\over 7},{10\over 3}}^{2,1}(\Omega^t)}+|\nabla p|_{{15\over 7},{10\over 3},\Omega^t}\cr
      &\le c(|f|_{{15\over 7},{10\over 3},\Omega^t}+\|v(0)\|_{B_{{15\over 7},{10\over 3}}^{2-6/10}(\Omega)}+\phi_1d_3)=d_4.\cr}
    \label{7.21}
  \end{equation}
  Lemma \ref{l7.3} yields
  \begin{equation}
    \|v\|_{L_\infty(0,t;B_{{15\over 7},{10\over 3}}^{2-6/10}(\Omega^t))}\le c\|v\|_{W_{{15\over 7},{10\over 3}}^{2,1}(\Omega^t)}.
    \label{7.22}
  \end{equation}
  Theorem 18.10 from \cite{BIN} gives
  \begin{equation}
    |v(t)|_{q,\Omega}\le c\|v\|_{B_{{15\over 7},{10\over 3}}^{7/5}(\Omega)}.
    \label{7.23}
  \end{equation}
  The estimate holds for any finite $q$ because it satisfies the relation $7/5\ge 7/5-3/q$.

  Next, we use the imbedding (see \cite[Ch. 3, Sect. 10]{BIN})
  \begin{equation}
    |\nabla v|_{5,\Omega^t}\le c\|v\|_{W_{{15\over 7},{10\over 3}}^{2,1}(\Omega^t)}.
    \label{7.24}
  \end{equation}
  From (\ref{7.23}) and (\ref{7.24}) we have
  \begin{equation}
    |v\cdot\nabla v|_{5',\Omega^t}\le cd_4^2,
    \label{7.25}
  \end{equation}
  where $5'<5$ but it is arbitrary close to 5.

  In view of (\ref{7.25}) and \cite{MS} we have
  \begin{equation}
    \|v\|_{W_{5'}^{2,1}(\Omega^t)}+|\nabla p|_{5',\Omega^t}\le c(|f|_{5',\Omega^t}+\|v(0)\|_{W_{5'}^{2-2/5'}(\Omega)}+d_4^2)\equiv d_5.
    \label{7.26}
  \end{equation}
  From (\ref{7.26}) it follows that $v\in L_\infty(\Omega^t)$ and $\nabla v\in L_q(\Omega^t)$ for any finite $q$.

  Then
  $$
    |\nabla(v'\cdot\nabla v)|_{5',\Omega^t}\le cd_5^2
  $$
  and $$
    |\partial_t^{1/2}(v'\cdot\nabla v)|_{5',\Omega^t}\le cd_5^2,
  $$
  where $\partial_t^{1/2}$ denotes the fractional partial derivative in time.

  Then \cite{MS} implies
  \begin{equation}\eqal{
      &\|v\|_{W_{10/3}^{3,3/2}(\Omega^t)}+\|\nabla p\|_{W_{10/3}^{1,1/2}(\Omega^t)}\cr
      &\le c(\|f\|_{W_{10/3}^{1,1/2}(\Omega^t)}+\|v(0)\|_{W_{10/3}^{3-2/5'}(\Omega)}+d_5^2)\equiv d_6.\cr}
    \label{7.27}
  \end{equation}
  Continuing the considerations yields
  \begin{equation}
    \|v\|_{W_2^{4,2}(\Omega^t)}+\|\nabla p\|_{W_2^{2,1}(\Omega^t)}\le c(\|f\|_{W_2^{2,1}(\Omega^t)}+\|v(0)\|_{W_2^3(\Omega)}+d_6^2).
    \label{7.28}
  \end{equation}
  This implies (\ref{7.4}) and ends the proof.
\end{proof}

\section{Global existence of solutions to problem (\ref{1.1})--(\ref{1.3}) for small data}\label{s8}

\begin{lemma}\label{l8.1}
  Assume that $G^2=|F_1|_{6/5,2,\Omega^t}^2+|\omega_1(0)|_{2,\Omega}^2+|r^2f_\varphi^2|_{2,\Omega^t}^2+|{v_\varphi^2(0)\over r}|_{2,\Omega}^2$, $G_1=|f_\varphi|_{\infty,1,\Omega^t}+|v_\varphi(0)|_{\infty,\Omega}$ are finite. Let $F_1={F_\varphi\over r}$, $\omega_1={\omega_\varphi\over r}$. Let $G_2=cG^3+G_1$.\\
  Let the following restriction hold $G_2\le{1\over[c(D_1^2+D_8^2)+1]^{3/2}}$, where $D_1$, $D_8$ are defined in Notation \ref{n1.1}. Then there exists a solution to (\ref{1.1})--(\ref{1.3}) such that
  \begin{equation}
    |v_\varphi|_{\infty,\Omega^t}\le{1\over[c(D_1^2+D_8^2)+1]^{1/2}}
    \label{8.1}
  \end{equation}
\end{lemma}

\begin{proof}
  Consider equations $(\ref{1.7})_2$ and $(\ref{1.8})_2$. Introducing the quantities
  \begin{equation}
    u_1={v_\varphi\over r},\quad \omega_1={\omega_\varphi\over r},\quad f_1={f_\varphi\over r},\quad F_1={F_\varphi\over r}.
    \label{8.2}
  \end{equation}
  We can write $(\ref{1.7})_2$ and $(\ref{1.8})_2$ in the form
  \begin{equation}
    u_{1,t}+v\cdot\nabla u_1-\nu\bigg(\Delta+{2\over r}\partial_r\bigg)u_1=-{1\over r}v_ru_1+f_1,
    \label{8.3}
  \end{equation}
  \begin{equation}
    \omega_{1,t}+v\cdot\nabla\omega_1-\nu\bigg(\Delta+{2\over r}\partial_r\bigg)\omega_1=2u_1u_{1,z}+F_1.
    \label{8.4}
  \end{equation}
  We proved in \cite{NZ1} the inequality
  \begin{equation}\eqal{
    &{1\over 4}\bigg|{v_\varphi^2(t)\over r}\bigg|_{2,\Omega}^2+{3\over 4}\nu\bigg|\nabla{v_\varphi^2\over r}\bigg|_{2,\Omega^t}^2+{\nu\over 2}\intop_{\Omega^t}{v_\varphi^4\over r^4}dxdt'\cr
    &\le{3\over 2}\intop_{\Omega^t}\bigg|{v_r\over r}\bigg|{v_\varphi^4\over r^2}dxdt'+{27\over 4\nu^3}\intop_{\Omega^t}f_\varphi^4r^4dxdt'+{1\over 4}\intop_\Omega{v_\varphi^2(0)\over r}dx.\cr}
    \label{8.5}
  \end{equation}
  Multiplying (\ref{8.4}) by $\omega_1$, integrating over $\Omega$, using boundary conditions and applying the H\"older and Young inequalities to the terms from the r.h.s. we obtain
  \begin{equation}
    |\omega_1(t)|_{2,\Omega}^2+\nu|\nabla\omega_1|_{2,\Omega^t}^2\le{2\over\nu}\bigg|{v_\varphi\over r}\bigg|_{4,\Omega^t}^4+{2\over\nu}|F_1|_{6/5,2,\Omega^t}^2+|\omega_1(0)|_{2,\Omega}^2.
    \label{8.6}
  \end{equation}Multiplying (\ref{8.6}) by ${\nu^2\over 8}$ and adding the resulting equation to (\ref{8.5}) yields
  \begin{equation}\eqal{
    &{1\over 4}\bigg|{v_\varphi^2(t)\over r}\bigg|_{2,\Omega}^2+{3\over 4}\nu\bigg|\nabla{v_\varphi^2\over r}\bigg|_{2,\Omega^t}^2+{\nu\over 4}\bigg|{v_\varphi\over r}\bigg|_{4,\Omega^t}^4+{\nu^2\over 8}(\omega_1(t)|_{2,\Omega}^2\cr
    &\quad+{\nu^3\over 8}|\nabla\omega_1|_{2,\Omega^t}^2\le{3\over 2}\intop_{\Omega^t}\bigg|{v_r\over r}\bigg|{v_\varphi^4\over r^2}dxdt'+{\nu\over 4}|f_1|_{6/5,2,\Omega^t}^2\cr
    &\quad+{\nu^2\over 8}|\omega_1(0)|_{2,\Omega}^2+{27\over 4\nu^3}|rf_\varphi|_{4,\Omega^t}^4+{1\over 4}\bigg|{v_\varphi^2(0)\over r}\bigg|_{2,\Omega}^2\cr
    &\le{3\over 2}\intop_{\Omega^t}\bigg|{v_r\over r}\bigg|{v_\varphi^4\over r^2}dxdt'+D_8^2G^2,\cr}
    \label{8.7}
  \end{equation}
  where
  $$\eqal{
    &D_8^2=\max\bigg\{{\nu\over 4},{\nu^2\over 8},{27\over 4\nu^3},{1\over 4}\bigg\}\cr
    &G^2=|F_1|_{6/5,2,\Omega^t}^2+|\omega_1(0)|_{2,\Omega}^2+|rf_\varphi|_{4,\Omega^t}^4+\bigg|{v_\varphi^2(0)\over r}\bigg|_{2,\Omega}^2.\cr}
  $$
  The first term on the r.h.s. of (\ref{8.7}) is bounded by
  $$
    {\varepsilon\over 2}\intop_{\Omega^t}{v_\varphi^4\over r^4}dx+{1\over 2\varepsilon}\bigg|{v_r\over r}\bigg|_{2,\Omega^t}^2|v_\varphi|_{\infty,\Omega^t}^2.
  $$
  Using the energy estimate
  $$
    \bigg|{v_r\over r}\bigg|_{2,\Omega^t}\le D_1
  $$
  we can write (\ref{8.7}) in the short form
  \begin{equation}
    {\nu^2\over 8}|\omega_1(t)|_{2,\Omega}^2+{\nu^3\over 8}|\nabla\omega_1|_{2,\Omega^t}^2+{\nu\over 8}\bigg|{v_\varphi\over r}\bigg|_{4,\Omega^t}^4\le D_1^2|v_\varphi|_{\infty,\Omega^t}^2+D_8^2G^2.
    \label{8.8}
  \end{equation}
  Multiplying $(\ref{1.7})_2$ by $v_\varphi|v_\varphi|^{s-2}$, integrate over $\Omega$ and use the boundary conditions we have
  $$
    {1\over s}{d\over dt}|v_\varphi|_{s,\Omega}^s+\nu\intop_\Omega{|v_\varphi|^s\over r^2}dx\le\intop_\Omega\bigg|{v_r\over r}\bigg|\,|v_\varphi|^sdx+\intop_\Omega f_\varphi v_\varphi|v_\varphi|^{s-2}dx.
  $$
  The first term on the r.h.s. is bounded by
  $$
    {\nu\over 2}\intop_\Omega{|v_\varphi|^s\over r^2}dx+{1\over 2\nu}\intop_\Omega v_r^2|v_\varphi|^sdx.
  $$
  Hence, we obtain the inequality
  \begin{equation}\eqal{
    &{d\over dt}|v_\varphi|_{s,\Omega}\le{1\over 2\nu}|v_r|_{\infty,\Omega}^2|v_\varphi|_{s,\Omega}+|f_\varphi|_{s,\Omega}\cr
    &\le{R^2\over 2\nu}\bigg|{v_r\over r}\bigg|_{\infty,\Omega}^2|v_\varphi|_{s,\Omega}+|f_\varphi|_{s,\Omega}\le c_*|\omega_1|_{2,\Omega}^2|v_\varphi|_{s,\Omega}+|f_\varphi|_{s,\Omega},\cr}
    \label{8.9}
  \end{equation}
  where $c_*$ is from the imbedding $|v_r/r|_{\infty,\Omega}\le c_*|\omega_1|_{2,\Omega}$.

  Integrating (\ref{8.9}) with respect to time yields
  $$
    |v_\varphi(t)|_{s,\Omega}\le c_*|\omega_1|_{2,\infty,\Omega^t}^2|v_\varphi|_{s,\infty,\Omega^t}+|f_\varphi|_{s,\Omega^t}+|v_\varphi(0)|_{s,\Omega}.
  $$
  Passing with $s\to\infty$ and using (\ref{8.8}) we obtain
  \begin{equation}\eqal{
      |v_\varphi|_{\infty,\Omega^t}&\le c_*8/\nu^2(D_1^2|v_\varphi|_{\infty,\Omega^t}^2+D_8^2G^2)|v_\varphi|_{\infty,\Omega^t}\cr
      &\quad+|f_\varphi|_{\infty,1,\Omega^t}+|v_\varphi(0)|_{\infty,\Omega}\le c(D_1^2+D_8^2)|v_\varphi|_{\infty,\Omega^t}^3+cG^3+G_1,\cr}
    \label{8.10}
  \end{equation}
  where $G_1=|f_\varphi|_{\infty,1,\Omega^t}+|v_\varphi(0)|_{\infty,\Omega}$.

  Let $M$ be a solution to the equations
  \begin{equation}
    M=c(D_1^2+D_8^2)M^3+G_2
    \label{8.11}
  \end{equation}
  where $G_2=cG^3+G_1$. Hence
  \begin{equation}
    |v_\varphi|_{\infty,\Omega^t}\le M.
    \label{8.12}
  \end{equation}
  We prove existence of solutions to (\ref{8.11}) by the following method of successive approximations
  \begin{equation}
    M_{n+1}=c(D_1^2+D_8^2)M_n^3+G_2.
    \label{8.13}
  \end{equation}
  First we show the uniform boundedness. Let $M_n\le A$. Then we need
  \begin{equation}
    c(D_1^2+D_8^2)A^3+G_2\le A.
    \label{8.14}
  \end{equation}
  Let $A=G_2^{1/3}$. Then
  \begin{equation}
    [c(D_1^2+D_8^2)+1]G_2\le G_2^{1/3}\quad {\rm so}\quad G_2\le{1\over[c(D_1^2+D_8^2)+1]^{3/2}}.
    \label{8.15}
  \end{equation}
  The above condition is a restriction on data. Therefore,
  \begin{equation}
    M_n\le G^{1/3}\quad {\rm for\ any}\ \ n\in\N.
    \label{8.16}
  \end{equation}
  Next we show convergence. Let $N_n=M_n-M_{n-1}$. Then (\ref{8.13}) implies
  \begin{equation}\eqal{
      &N_{n+1}\le c(D_1^2+D_8^2)(M_n^2+M_nM_{n-1}+M_{n-1}^2)N_n\cr
      &\le c(D_1^2+D_8^2){1\over[c(D_1^2+D_8^2)+1]^3}N_n\cr}
    \label{8.17}
  \end{equation}
  We have convergence if
  \begin{equation}
    {c(D_1^2+D_8^2)\over[c(D_1^2+D_8^2)+1]^3}<1.
    \label{8.18}
  \end{equation}
  Hence, (\ref{8.17}) holds. Therefore, we proved that
  \begin{equation}
    |v_\varphi|_{\infty,\Omega^t}\le{1\over[c(D_1^2+D_8^2)+1]^{1/2}}=\colon{1\over c_1}
    \label{8.19}
  \end{equation}
  Therefore, we have existence of global regular solutions with small data.
\end{proof}

This small-data result is independent of the critical-wedge obstruction and is included as a separate regime in which the angular component is controlled directly.

\noindent
{\bf Conflict of interest statement}

The authors report there are no competing interests to declare.

\noindent
{\bf Data availability statement}

The authors report that there is no data associated with this work.

\bibliographystyle{amsplain}

\end{document}